\documentclass[a4paper,11pt]{amsart}

\usepackage{amsfonts,latexsym,rawfonts,amsmath,amssymb,amsthm, mathrsfs, lscape}
\usepackage[all]{xy}
\usepackage[english]{babel}
\usepackage{array, tabularx}
\usepackage{setspace}
\usepackage{textcomp}
\usepackage{url}
\usepackage{enumerate}
\usepackage[utf8x]{inputenc}
\usepackage[T1]{fontenc}
\usepackage{graphicx}
\usepackage{booktabs}
\usepackage{multirow}
\usepackage{array}
\usepackage{fullpage}

\usepackage[hypertexnames=false,
backref=page,
    pdftex,
    pdfpagemode=UseNone,
    breaklinks=true,
    extension=pdf,
    colorlinks=true,
    linkcolor=blue,
    citecolor=blue,
    urlcolor=blue,
]{hyperref}

\newtheorem{theorem}{Theorem}[section]
\newtheorem{corollary}[theorem]{Corollary}
\newtheorem{assumption}[theorem]{Assumption}
\newtheorem{remark}[theorem]{Remark}
\newtheorem{proposition}[theorem]{Proposition}
\newtheorem{lemma}[theorem]{Lemma}

\newcommand{\N}{\mathbb{N}}
\newcommand{\Z}{\mathbb{Z}}

\newcommand{\R}{\mathbb{R}}
\newcommand{\bC}{\mathbb{C}}
\newcommand{\n}{\mathfrak{n}}
\newcommand{\g}{\mathfrak{g}}
\newcommand{\solvmfd}{\left. \Gamma \middle\backslash G \right.}
\newcommand{\del}{\partial}
\newcommand{\delbar}{\overline{\del}}

\DeclareMathOperator{\imm}{im}
\DeclareMathOperator{\de}{d}
\DeclareMathOperator{\esp}{e}
\DeclareMathOperator{\GL}{GL}
\DeclareMathOperator{\Aut}{Aut}
\DeclareMathOperator{\Hom}{Hom}
\DeclareMathOperator{\Ad}{Ad}
\DeclareMathOperator{\diag}{diag}

\title[Cohomologies of deformations of solvmanifolds]{Cohomologies of deformations of solvmanifolds and closedness of some properties}

\author{Angella, Daniele}
\address[Daniele Angella]{Dipartimento di Matematica e Informatica "Ulisse Dini", Università degli Studi di Firenze, viale Morgagni 67/a, 50134 Firenze, Italy}
\email{daniele.angella@gmail.com}
\email{daniele.angella@unifi.it}

\author{Kasuya, Hisashi}
\address[Hisashi Kasuya]{Department of Mathematics, Graduate School of Science, Osaka University, Toyonaka, Osaka 560-0043, Japan}
\email{kasuya@math.sci.osaka-u.ac.jp}



\thanks{The first author has been granted with a research fellowship by Istituto Nazionale di Alta Matematica INdAM, and is supported by the Project PRIN ``Variet\`{a} reali e complesse: geometria, topologia e analisi armonica'', by the Project FIRB ``Geometria Differenziale e Teoria Geometrica delle Funzioni'',
and by GNSAGA of INdAM.
The second author is supported by JSPS Research Fellowships for Young Scientists.
\medskip
The first author is greatly indebted to Adriano Tomassini for his constant support and encouragement. He would like to thank also Gunnar \TH{}\'{o}r Magn\'{u}sson for useful discussions.
The authors thank Luis Ugarte for valuable discussions which motivated them to study closedness under holomorphic deformations more deeply.
Thanks also to the anonymous Referee for her/his suggestions on presentation.\\
The main part of the work was accomplished during the second author's stay at Dipartimento di Matematica of Universit\`{a} di Pisa, in the winter of 2013.
}

\begin{document}

\begin{abstract}
We provide further techniques to study the Dolbeault and Bott-Chern cohomologies of deformations of solvmanifolds by means of finite-dimensional complexes. 
By these techniques, we can compute the Dolbeault and Bott-Chern cohomologies of some complex solvmanifolds, and we also get explicit examples, showing in particular that either the $\partial\overline{\partial}$-Lemma or the property that the Hodge and Fr\"olicher spectral sequence degenerates at the first level are not closed under deformations.
\end{abstract}

\maketitle

\section*{Introduction}
Among other techniques, the theory of \emph{small deformations of holomorphic structures}, initiated and developed by K. Kodaira and D.~C. Spencer, L. Nirenberg, and M. Kuranishi, provides a large source of examples of compact complex manifolds.

As a natural problem, the behaviour of special metrics or cohomological properties under deformations deserves special interests in order to better understand the geometry of complex manifolds. In such a context, the stability results for K\"ahler structures plays a guiding role: in fact, K. Kodaira and D.~C. Spencer proved in \cite[Theorem 15]{kodaira-spencer-3} that any small deformations of a compact K\"ahler manifold still admits a K\"ahler metric. On the other hand, the result holds no more true when replacing the K\"ahler condition with weaker metric conditions, such as, for example, the existence of balanced metrics in the sense of M. L. Michelsohn, \cite[Proposition 4.1]{alessandrini-bassanelli}, or the existence of pluri-closed metrics, \cite[Theorem 2.2]{fino-tomassini-advmath}, (nor also in the non-elliptic context of $\mathbf{D}$-complex geometry in the sense of F.~R. Harvey and H.~B. Lawson, \cite[Theorem 4.2]{angella-rossi}, nor in the non-integrable case of almost-K\"ahler geometry). As 
regards cohomological properties, the stability of the $\partial\overline{\partial}$-Lemma under deformations has been proved in several ways, see \cite[Proposition 9.21]{voisin}, \cite[Theorem 5.12]{wu}, \cite[\S B]{tomasiello}, \cite[Corollary 2.7]{angella-tomassini-3}. (We recall that a compact complex manifold is said to satisfy the \emph{$\partial\overline{\partial}$-Lemma} if every $\partial$-closed, $\overline{\partial}$-closed, $\de$-exact form is also $\partial\overline{\partial}$-exact, see, e.g., \cite{deligne-griffiths-morgan-sullivan}.)
K. Kodaira and D.~C. Spencer's result, for example, can be phrased by saying that, for any family of compact complex manifolds parametrized over the manifold $\mathcal{B}$, the set of parameters of $\mathcal{B}$ for which the corresponding complex manifold admits a K\"ahler metric is open in the topology of $\mathcal{B}$. In \cite[Theorem 2.20]{angella-kasuya-1}, in studying the cohomologies of the completely-solvable Nakamura manifold, the authors provided an example of a curve $\left\{J_t\right\}_{t\in B}$ of complex structures and of a sequence $\left\{t_k\right\}_{k\in\N}\subset B$ converging to $t_\infty$ in the topology of $B$ such that $\left(X,\, J_{t_k}\right)$ satisfies the $\partial\overline{\partial}$-Lemma for any $k\in\N$ but $\left(X,\, J_{t_\infty}\right)$ does not; in other words, the set of parameters for which the $\partial\overline{\partial}$-Lemma holds is not closed in the (strong) topology of the base space. Actually, as L. Ugarte pointed out to us, in studying the behaviour under 
limits of compact complex manifolds, it is common to consider Zariski 
topology instead of strong topology: in fact, e.g., Mo\v{\i}{\v{s}}ezon property, \cite{moishezon}, is supposed to be closed with respect to the Zariski topology, see \cite{popovici-inventiones} for motivations and results, while it is not closed in the strong topology. With such a notion of (Zariski) closedness, we provide here an example to prove the following result. Note that the non-closedness of $E_{1}$-degeneration of the Hodge and Fr\"olicher spectral sequences was already proven by M.~G. Eastwood and M.~A. Singer in \cite[Theorem 5.4]{eastwood-singer} by using twistor spaces.

\smallskip
\noindent{\bfseries Theorem (see Corollary \ref{cor:non-closedness}).}
{\itshape
The property of $E_{1}$-degeneration of the Hodge and Fr\"olicher spectral sequences and the property of satisfying the $\partial\bar\partial$-Lemma are not closed under holomorphic deformations.
}
\smallskip

\medskip

In order to provide such an example, we continue in investigating the class of nilmanifolds and solvmanifolds from the point of view of cohomologies computations. More precisely, we would enlarge the class of solvmanifolds for which the de Rham, Dolbeault, and Bott-Chern cohomologies can be computed by means of just a finite-dimensional sub-complex of the double-complex of differential forms, by carrying over the results in \cite{nomizu, hattori, mostow, mostow-errata, guan, console-fino-sns, kasuya-jdg, kasuya-holpar, console-fino-kasuya, sakane, cordero-fernandez-gray-ugarte, console-fino, rollenske, 
rollenske-survey, kasuya-mathz, angella-1, angella-kasuya-1}. (We recall that, given a double-complex $\left(A^{\bullet,\bullet},\, \del,\, \delbar\right)$, the Dolbeault cohomology is $H^{\bullet,\bullet}_{\delbar}(A^{\bullet,\bullet}):=\frac{\ker\delbar}{\imm\delbar}$ and the \emph{Bott-Chern cohomology} is $H^{\bullet,\bullet}_{BC}(A^{\bullet,\bullet}):=\frac{\ker\del\cap\ker\delbar}{\imm\del\delbar}$; one can also consider the \emph{Aeppli cohomology}, $H^{\bullet,\bullet}_{A}(A^{\bullet,\bullet}):=\frac{\ker\del\delbar}{\imm\del+\imm\delbar}$, which is, in a sense, the dual of the Bott-Chern cohomology; finally, in considering a complex manifold, the Dolbeault and Bott-Chern cohomology are defined by means of the double-complex $\left(\wedge^{\bullet,\bullet}X,\, \del,\, \delbar\right)$ of complex-valued differential forms; see \cite{aeppli, bott-chern}, see also \cite{demailly-agbook, schweitzer}.) More precisely, we provide the following stability results for cohomology computations of deformations of 
solvmanifolds, in the vein of the results proven in \cite[Theorem 1]{console-fino} and \cite[Theorem 3.9]{angella-1} for nilmanifolds; (we refer to Theorem \ref{thm:dolb-deformations} and Theorem \ref{thm:bc-deformations} for the precise statement).

\smallskip
\noindent{\bfseries Theorem (see Theorem \ref{thm:dolb-deformations} and Theorem \ref{thm:bc-deformations}).}
{\itshape
Given a solvmanifold $X=\left.\Gamma\middle\backslash G\right.$ endowed with a left-invariant complex structure $J$, for which there exists a finite-dimensional sub-complex $C^{\bullet,\bullet}\subset \wedge^{\bullet,\bullet}X$ computing the Dolbeault cohomology, we provide conditions in order that suitable deformations of $C^{\bullet,\bullet}$, still allow to compute Dolbeault and Bott-Chern cohomologies of some small deformations of $J$.
}
\smallskip

The proof of this theorem is inspired by the proof of K. Kodaira and D.~C. Spencer's theorem on the upper-semi-continuity of the dimensions of the Dolbeault cohomology groups, \cite{kodaira-spencer-3}.
Considering downers of cohomologies, differing from upper-semi-continuity, by this theorem we can observe ``nose-diving'' phenomena, as in the following examples, which are generalizations of the three-dimensional examples found by K. Kodaira and I. Nakamura, \cite{nakamura}.

\smallskip
\noindent{\bfseries Example (see Section \ref{SYE}).}
{\itshape
Let $N$ be a complex nilpotent Lie group.
Suppose that the Lie algebra of $N$ has a $(\Z+\sqrt{-1}\,\Z)$-basis.
Then, for certain semidirect product $G=\bC\ltimes _{\phi}(N\times N)$, we have a lattice $\Gamma$ of $G$ by the results of H. Sawai and T. Yamada, {\rm \cite{sawai-yamada}}, and there exists a deformation $\{J_{t}\}_t$ of the holomorphically parallelizable solvmanifold $\solvmfd$ such that $\dim H^{1,0}_{\bar\partial_{t}}(\solvmfd)=0$, where $H^{\bullet,\bullet}_{\bar\partial_{t}}(\solvmfd)$ is the Dolbeault cohomology of a deformed complex solvmanifold.
}
\smallskip

In particular, it follows that the above examples provide a new class of ``Dolbeault-cohomologically-computable'' complex solvmanifolds, since they are not holomorphic fibre bundles over complex tori as in \cite{kasuya-mathz, kasuya-holpar, console-fino-kasuya}.

\section{Deformations and cohomology}

Let $(M,J)$ be a compact complex manifold and $\left(\wedge^{\bullet,\bullet}_{J}M,\, \partial,\, \bar\partial\right)$ be the double-complex of complex-valued differential forms on $M$ associated with the complex structure $J$.
We consider deformations $\{J_{t}\}_{t\in B}$ over a ball $B\subset \bC^{m}$ such that $J_0=J$.
We also consider the  double-complex $\left(\wedge^{\bullet,\bullet}_{J_{t}}M,\, \partial_{t},\, \bar\partial_{t}\right)$ associated with the deformed complex structure $J_{t}$.

We are interested in manifolds whose cohomologies can be computed by means of just a finite-dimensional sub-double-complex of $\left(\wedge^{\bullet,\bullet}_{J}M,\, \partial,\, \bar\partial\right)$. In particular, in this section, we are concerned in studying the behaviour of such a property under small deformations of the complex structure.

\medskip

Inspired by \cite{kodaira-spencer-3}, we prove the following result.
\begin{theorem}\label{thm:dolb-deformations}
Let $(M,J)$ be a compact complex manifold, and consider deformations $\left\{J_t\right\}_{t\in B}$ such that $J_0=J$.
We suppose that we have a family $\left\{ C^{\bullet,\bullet}_{t} = \bC\langle \phi^{\bullet,\bullet}_{i}(t)\rangle_{i} \right\}_{t \in B}$ of sub-vector spaces of $(\wedge^{\bullet,\bullet}_{J_t}M,\partial_t,\bar\partial_t)$ parametrized by $t\in B$ and spanned by linearly-independent vectors $\phi^{\bullet,\bullet}_{i}(t)$ so that:
\begin{enumerate}[(1)]
 \item for each $t\in B$, it holds that $(C^{\bullet,\bullet}_{t},\bar\partial_{t})$ is a sub-complex of $(\wedge^{\bullet,\bullet}_{J_{t}}M,\bar\partial_{t})$;
 \item $\phi^{\bullet,\bullet}_{i}(t)$ is smooth on $M\times  B$, for any $i$;
 \item the inclusion $C^{\bullet,\bullet}_{0}\subset \wedge^{\bullet,\bullet}_{J}M$ induces the cohomology isomorphism
\[ H^{\bullet,\bullet}_{\bar\partial_0}(C_{0}^{\bullet,\bullet})\cong H^{\bullet,\bullet}_{\bar\partial}(M) \;; \]
 \item there exists a smooth family $\{g_{t}\}_{t\in B}$ of $J_t$-Hermitian metrics such that $\bar*_{g_{t}}(C^{\bullet,\bullet}_t)\subseteq C^{n-\bullet,n-\bullet}_t$, where we denote by $\bar*_{g_{t}}$ the anti-$\bC$-linear Hodge-$*$-operator of $g_{t}$, and by $2n$ the real dimension of $M$.
\end{enumerate}
Then, for sufficiently small $t$, the inclusion $C^{\bullet,\bullet}_{t}\subset \wedge^{\bullet,\bullet}_{J_{t}}(M)$ 
 induces the cohomology isomorphism
\[ H^{\bullet,\bullet}_{\bar\partial_{t}}(C_t^{\bullet,\bullet})\cong H^{\bullet,\bullet}_{\bar\partial_{t}}(M) \;. \]
\end{theorem}

\begin{proof}
Consider the operators $\bar\partial^{*}_{t}=-\bar*_{g_{t}}\bar\partial_{t}\bar*_{g_{t}}$
and $\Delta_{\bar\partial_{t}}=\bar\partial_{t}\bar\partial^{*}_{t}+\bar\partial^{*}_{t}\bar\partial_{t}$.
Then by the conditions {\itshape (1)} and {\itshape (4)}, the operator $\Delta_{\bar\partial_{t}}$ can be defined on $C^{\bullet,\bullet}_{t}$.
By a result by K. Kodaira and D.~C. Spencer, \cite[Theorem 11]{kodaira-spencer-3}, see also \cite[Theorem 7.1]{kodaira}, for each $t\in B$, we have a basis $\{e_{1}(t), \ldots, e_{i}(t), \ldots\}$ of $\wedge^{\bullet,\bullet}_{J_{t}}M$ and continuous functions $a_{1}(t)\le \dots \le a_{i}(t) \le \cdots$ on $B$ such that 
$\Delta_{\bar\partial_{t}} e_{i}(t) =a_{i}(t)e_{i}(t)$ for any $i$.
Since $\Delta_{\bar\partial_{t}}$ is defined on $C^{\bullet,\bullet}_{t}$,
we can take a subset $\left\{ e_{i_{1}}(t), \dots, e_{i_{\ell}}(t) \right\}$ of $\{e_{i}(t)\}_i$ that is a basis of $C^{\bullet,\bullet}_{t}$.
Take $\{e_{j}(t),\dots, e_{j+k}(t)\}=\{e_{i}(t) \;\vert\; a_{i}(0)=0\}$.
Then $\{e_{j}(0),\dots, e_{j+k}(0)\}$ is a basis of $\ker \Delta_{\bar\partial_{0}}$.
By the assumption {\itshape (3)}, we have $\ker \Delta_{\bar\partial_{0}}\subseteq C^{\bullet,\bullet}_0$.
Hence we have $\{e_{j}(t),\dots, e_{j+k}(t)\}\subseteq C^{\bullet,\bullet}_{t}$ for any $t \in B$.
Since each $a_{i}$ is continuous, we have, for sufficiently small $t\in B$, that $a_{j-1}(t)<0$ and $0<a_{j+k+1}(t)$.
Hence we have  $\ker \Delta_{\bar\partial_{t}}\subseteq \{e_{j}(t),\dots, e_{j+k}(t)\}\subseteq C^{\bullet,\bullet}_{t}$.
Hence the theorem follows.
\end{proof}

Analogously, as regards the Bott-Chern cohomology, by considering the operators $\bar\partial^{*}_{t}=-\bar*_{g_{t}}\bar\partial_{t}\bar*_{g_{t}}$ and $\partial^{*}_{t}=-\bar*_{g_{t}}\partial_{t}\bar*_{g_{t}}$, and $\tilde\Delta_{{BC}_{t}} =
\partial_t\bar\partial_t\bar\partial_t^*\partial_t^*+\bar\partial_t^*\partial_t^*\partial_t\bar\partial_t+\bar\partial_t^*\partial_t\partial_t^*\bar\partial_t+\partial_t^*\bar\partial_t\bar\partial_t^*\partial_t+\bar\partial_t^*\partial_t+\partial_t^*\partial_t$, see \cite[Proposition 5]{kodaira-spencer-3} and \cite[\S2.b]{schweitzer}, a similar argument yields the following result.

\begin{theorem}\label{thm:bc-deformations}
Let $(M,J)$ be a compact complex manifold, and consider deformations $\left\{J_t\right\}_{t\in B}$ such that $J_0=J$.
We suppose that we have a family $\left\{ C^{\bullet,\bullet}_{t} = \bC\langle \phi^{\bullet,\bullet}_{i}(t)\rangle_{i} \right\}_{t\in B}$ of sub-vector spaces of $(\wedge^{\bullet,\bullet}_{J_t}M,\partial_t,\bar\partial_t)$ parametrized by $t\in B$ and spanned by linearly-independent vectors $\phi^{\bullet,\bullet}_{i}(t)$ so that:
\begin{enumerate}[(1)]
 \item for each $t\in B$, it holds that $(C^{\bullet,\bullet}_{t}, \partial_{t}  ,\bar\partial_{t})$ is a sub-double-complex of  $(\wedge^{\bullet,\bullet}_{J_{t}}M,\partial_{t},\bar\partial_{t})$;
 \item $\phi^{\bullet,\bullet}_{i}(t)$ is smooth on $M\times B$, for any $i$;
 \item the inclusion $C^{\bullet,\bullet}_{0}\subset \wedge^{\bullet,\bullet}_{J}M$ induces the Bott-Chern cohomology isomorphism
\[ H^{\bullet,\bullet}_{BC}(C_{0}^{\bullet,\bullet}) \cong H^{\bullet,\bullet}_{BC}(M) \;; \]
 \item there exists a smooth family $\{g_{t}\}_{t\in B}$ of $J_t$-Hermitian metrics such that $\bar*_{g_{t}}(C^{\bullet,\bullet}_t)\subseteq C^{n-\bullet,n-\bullet}_t$, where we denote by $\bar*_{g_{t}}$ the anti-$\bC$-linear Hodge-$*$-operator of $g_{t}$, and by $2n$ the real dimension of $M$.
\end{enumerate}
Then, for sufficiently small $t$, the inclusion $C^{\bullet,\bullet}_{t}\subset \wedge^{\bullet,\bullet}_{J_{t}}M$ induces the Bott-Chern cohomology isomorphism
\[ H^{\bullet,\bullet}_{BC}(C_{t}^{\bullet,\bullet})\cong H^{\bullet,\bullet}_{BC}(M) \;. \]
\end{theorem}

\section{Applications: nilmanifolds}

Consider nilmanifolds, that is, compact quotients of connected simply-connected nilpotent Lie groups by discrete co-compact subgroups, and take left-invariant complex structures.
By considering the sub-double-complex $C^{\bullet,\bullet} = \wedge^{\bullet,\bullet}(\mathfrak{g}\otimes_\R\bC)^*$ of left-invariant differential forms, where $\mathfrak{g}$ is the Lie algebra associated to the nilmanifold, one recovers the stability results in \cite{console-fino, angella-1} by Theorem \ref{thm:dolb-deformations} and Theorem \ref{thm:bc-deformations}.

\begin{corollary}[{\cite[Theorem 1]{console-fino}, \cite[Theorem 3.9]{angella-1}}]
Let $X = \left. \Gamma \middle\backslash G \right.$ be a nilmanifold, and denote the Lie algebra associated to $G$ by $\mathfrak{g}$ and its complexification by $\mathfrak{g}_\bC := \mathfrak{g}\otimes_\R\bC$. The set of $G$-left-invariant complex structures on $X$ such that the inclusion $\wedge^{\bullet,\bullet}\mathfrak{g}_\bC^* \subset \wedge^{\bullet,\bullet}X$ induces the isomorphism $H^{\bullet,\bullet}_{\bar\partial}(\wedge^{\bullet,\bullet}\mathfrak{g}^*) \cong H^{\bullet,\bullet}_{\bar\partial}(X)$,
respectively $H^{\bullet,\bullet}_{BC}(\wedge^{\bullet,\bullet}\mathfrak{g}^*) \cong H^{\bullet,\bullet}_{BC}(X)$,
is open in the set of $G$-left-invariant complex structures on $X$.
\end{corollary}

We recall that, in view of \cite[Theorem 1]{sakane}, \cite[Main Theorem]{cordero-fernandez-gray-ugarte}, \cite[Theorem 2, Remark 4]{console-fino}, \cite[Theorem 1.10]{rollenske}, and \cite[Corollary 3.10]{rollenske-survey}, \cite[Theorem 3.8]{angella-1}, the above set contains several classes of left-invariant complex structures, among which holomorphically parallelizable, Abelian, nilpotent, and rational.

\section{Applications: solvmanifolds}

In order to investigate explicit examples, we recall some results concerning the computations of Dolbeault cohomology for solvmanifolds of two special classes, namely, \emph{solvmanifolds of splitting-type}, (that is, satisfying Assumption \ref{ass:solvmanifolds},) \cite{kasuya-mathz}, and \emph{holomorphically parallelizable solvmanifolds}, (that is, with holomorphically-trivial holomorphic tangent bundle,) \cite{wang}.

\medskip

We start by considering solvmanifolds of the following type, see \cite{kasuya-mathz}. We call them \emph{solvmanifolds of splitting-type}.

\begin{assumption}\label{ass:solvmanifolds}
Consider a solvmanifold $X = \solvmfd$ endowed with a $G$-left-invariant complex structure $J$. Assume that $G$ is the semi-direct product $\bC^{n}\ltimes_{\phi}N$ so that:
\begin{enumerate}
 \item\label{item:ass-1} $N$ is a connected simply-connected $2m$-dimensional nilpotent Lie group endowed with an $N$-left-invariant complex structure $J_N$;
 \item\label{item:ass-2} for any $t\in \bC^{n}$, it holds that $\phi(t)\in \GL(N)$ is a holomorphic automorphism of $N$ with respect to $J_N$;
 \item\label{item:ass-3} $\phi$ induces a semi-simple action on the Lie algebra $\n$ associated to $N$;
 \item\label{item:ass-4} $G$ has a lattice $\Gamma$; (then $\Gamma$ can be written as $\Gamma = \Gamma_{\bC^n} \ltimes_{\phi} \Gamma_{N}$ such that $\Gamma_{\bC^n}$ and $\Gamma_{N}$ are  lattices of $\bC^{n}$ and, respectively, $N$, and, for any $t\in \Gamma^{\prime}$, it holds $\phi(t)\left(\Gamma_N\right)\subseteq\Gamma_N$;)
 \item\label{item:ass-5} the inclusion $\wedge^{\bullet,\bullet}\left(\n\otimes_\R\bC\right)^* \hookrightarrow \wedge^{\bullet,\bullet}\left(\left. \Gamma_{N} \right\backslash N\right)$ induces the isomorphism
$$ H^{\bullet,\bullet}_{\bar\partial}\left(\wedge^{\bullet,\bullet}\left(\n\otimes_\R\bC\right)^*\right) \stackrel{\cong}{\to} H^{\bullet,\bullet}_{\bar\partial }\left(\left.\Gamma_{N}\middle\backslash N\right.\right) \;. $$
\end{enumerate}
\end{assumption}

Consider the standard basis $\left\{ X_{1},\, \dots,\, X_{n} \right\}$ of $\bC^n$.
Consider the decomposition $\n\otimes_\R{\bC}=\n^{1,0}\oplus \n^{0,1}$ induced by $J_N$.
By the condition {\itshape (\ref{item:ass-2})}, this decomposition is a direct sum of $\bC^{n}$-modules.
By the condition {\itshape (\ref{item:ass-3})}, we have a basis $\left\{Y_{1},\, \dots,\, Y_{m}\right\}$ of $\n^{1,0}$ and characters $\alpha_1,\ldots,\alpha_m\in\Hom(\bC^n;\bC^*)$ such that the induced action $\phi$ on $\n^{1,0}$ is represented by
$$ \bC^n \ni t \mapsto \phi(t) \;=\; \diag \left( \alpha_{1}(t),\, \dots,\, \alpha_{m} (t) \right) \in \GL(\n^{1,0}) \;. $$
For any $j\in\{1,\ldots,m\}$, since $Y_{j}$ is an $N$-left-invariant $(1,0)$-vector field on $N$,
the $(1,0)$-vector field $\alpha_{j}Y_{j}$ on $\bC^{n}\ltimes _{\phi} N$ is $\left(\bC^{n}\ltimes _{\phi} N\right)$-left-invariant.
Consider the Lie algebra $\g$ of $G$ and the decomposition $\g_\bC := \g \otimes_\R \bC = \g^{1,0} \oplus \g^{0,1}$ induced by $J$.
Hence we have a basis $\left\{ X_{1},\, \dots,\, X_{n},\, \alpha_{1}Y_{1},\, \dots,\, \alpha_{m}Y_{m}\right\}$ of $\g^{1,0}$, and let $\left\{ x_{1},\, \dots,\, x_{n},\, \alpha^{-1}_{1}y_{1},\, \dots,\, \alpha_{m}^{-1}y_{m} \right\}$ be its dual basis of $\wedge^{1,0}\g_\bC^*$.
Then we have 
$$ \wedge^{p,q}\g_\bC^* \;=\; \wedge^{p} \left\langle x_{1},\, \dots,\, x_{n},\, \alpha^{-1}_{1}y_{1},\, \dots,\,\alpha^{-1}_{m}y_{m} \right\rangle \otimes \wedge^{q} \left\langle \bar x_{1},\, \dots,\, \bar x_{n},\, \bar\alpha^{-1}_{1}\bar y_{1},\, \dots,\, \bar\alpha^{-1}_{m}\bar y_{m} \right\rangle \;.$$

\medskip

The following lemma holds.

\begin{lemma}[{\cite[Lemma 2.2]{kasuya-mathz}}]\label{charr}
Let $X = \solvmfd$ be a solvmanifold endowed with a $G$-left-invariant complex structure $J$ as in Assumption \ref{ass:solvmanifolds}.
With the above notations,
for any $j\in\{1,\ldots,m\}$, there exist unique unitary characters $\beta_{j}\in\Hom(\bC^n;\bC^*)$ and $\gamma_{j}\in\Hom(\bC^n;\bC^*)$ on $\bC^{n}$ such that $\alpha_{j}\beta_{j}^{-1}$ and $\bar\alpha_{j}\gamma^{-1}_{j}$ are holomorphic.
\end{lemma}

Hence, define the differential bi-graded sub-algebra $B^{\bullet,\bullet}_{\Gamma} \subset \wedge^{\bullet,\bullet} \solvmfd$, for $(p,q)\in\Z^2$, as
\begin{eqnarray}\label{eq:def-b}
B^{p,q}_{\Gamma} &:=& \bC\left\langle x_{I} \wedge \left(\alpha^{-1}_{J}\beta_{J}\right)\, y_{J} \wedge \bar x_{K} \wedge \left( \bar\alpha^{-1}_{L}\gamma_{L} \right) \, \bar y_{L}
\;\middle\vert\;
\left|I\right| + \left|J\right|=p \text{ and } \left|K\right| + \left|L\right| = q \right. \\[5pt]
\nonumber
&& \left. \text{ such that } \left.\left( \beta_{J}\gamma_{L} \right)\right\lfloor_{\Gamma} = 1
\right\rangle
\end{eqnarray}
(where we shorten, e.g., $\alpha_I:=\alpha_{i_1}\cdot\cdots\cdot\alpha_{i_k}$ and $x_I:=x_{i_1}\wedge\cdots\wedge x_{i_k}$ for a multi-index $I=\left(i_1,\ldots,i_k\right)$ of length $|I|=k$).

We recall the following result by the second author.

\begin{theorem}{\rm (\cite[Corollary 4.2]{kasuya-mathz})}\label{CORR}
Let $X = \solvmfd$ be a solvmanifold endowed with a $G$-left-invariant complex structure $J$ as in Assumption \ref{ass:solvmanifolds}.
Consider the differential bi-graded sub-algebra $B^{\bullet,\bullet}_{\Gamma} \subset \wedge^{\bullet,\bullet} \solvmfd$ defined in \eqref{eq:def-b}.
Then the inclusion $B^{\bullet,\bullet}_{\Gamma}\subset \wedge^{\bullet,\bullet}\solvmfd$ induces the cohomology isomorphism
\[
H^{\bullet,\bullet}_{\bar\partial}\left(B^{\bullet,\bullet}_{\Gamma}\right) \stackrel{\cong}{\to} H^{\bullet,\bullet}_{\bar\partial}\left(\solvmfd\right) \;.
\]
\end{theorem}

As regards the Bott-Chern cohomology, define $\bar B^{\bullet,\bullet}_{\Gamma} := \left\{ \bar \omega \in \wedge^{\bullet,\bullet}\solvmfd \;\vert\; \omega\in B^{\bullet,\bullet}_{\Gamma} \right\}$ and
\begin{equation}\label{eq:def-c}
C^{\bullet,\bullet}_{\Gamma} \;:=\; B^{\bullet,\bullet}_{\Gamma} + \bar B_{\Gamma}^{\bullet,\bullet} \;.
\end{equation}
The authors proved the following result.

\begin{theorem}[{\cite[Theorem 2.16]{angella-kasuya-1}}]\label{BCISO}
Let $\solvmfd$ be a solvmanifold endowed with a $G$-left-invariant complex structure $J$ as in Assumption \ref{ass:solvmanifolds}.
Consider $C^{\bullet,\bullet}_{\Gamma}$ as in \eqref{eq:def-c}.
Then the inclusion $C^{\bullet,\bullet}_{\Gamma}\subset \wedge^{\bullet,\bullet}\solvmfd$ induces the isomorphisms
\[
H^{\bullet,\bullet}_{\bar\partial}\left(C^{\bullet,\bullet}_{\Gamma}\right) \stackrel{\cong}{\to} H^{\bullet,\bullet}_{\bar\partial}\left(\solvmfd\right)
\qquad \text{ and } \qquad
H^{\bullet,\bullet}_{BC}(C_\Gamma^{\bullet,\bullet}) \stackrel{\cong}{\to} H^{\bullet,\bullet}_{BC}\left(\solvmfd\right) \;.
\]
\end{theorem}

\medskip

Another class of ``cohomologically-computable'' solvmanifolds is given by \emph{holomorphically parallelizable solvmanifolds}, namely, compact quotients of connected simply-connected complex solvable Lie groups by co-compact discrete subgroups, \cite{wang}, see also \cite{nakamura}.

Let $G$ be a connected simply-connected complex solvable Lie group admitting a lattice $\Gamma$, and denote by $2n$ the real dimension of $G$. Denote the Lie algebra naturally associated to $G$ by $\g$. 

Denote by $ \g_{+}$ (respectively, $ \g_{-}$) the Lie algebra of the $G$-left-invariant holomorphic (respectively, anti-holomorphic) vector fields on $G$.
As a (real) Lie algebra, we have an isomorphism  $ \g_{+} \cong \g_{-}$ by means of the complex conjugation.

Let $N$ be the nilradical of $G$.
We can take a connected simply-connected complex nilpotent subgroup $C\subseteq G$  such that $G=C\cdot N$, see, e.g., \cite[Proposition 3.3]{dek}.
Since $C$ is nilpotent, the map
\[C\ni c \mapsto ({\Ad}_{c})_{\mathrm{s}}\in {\Aut}(\g_{+})\]
is a homomorphism, where $({\Ad}_{c})_{\mathrm{s}}$ is the semi-simple part of the Jordan decomposition of ${\Ad}_{c}$.
We have a basis $\left\{ X_{1},\dots,X_{n} \right\}$ of $\g_{+}$ such that, for $c \in C$,
$$ ({\Ad}_{c})_{\mathrm{s}} \;=\; {\diag} \left(\alpha_{1}(c),\dots,\alpha_{n}(c) \right) \;, $$
for some characters $\alpha_1,\dots,\alpha_n$ of $C$.
By $G=C\cdot N$, we have $G/N=C/C\cap N$ and regard $\alpha_1,\dots,\alpha_n$ as characters of $G$.
Let $\left\{ x_{1},\dots, x_{n} \right\}$ be the basis of $\g^{*}_{+}$ which is dual to $\left\{ X_{1},\dots ,X_{n} \right\}$.

Let $B^{\bullet}_{\Gamma}$ be the sub-complex of $\left( \wedge^{0,\bullet} \solvmfd, \, \delbar \right) $ defined as
\begin{equation}\label{eq:def-b-holpar}
B^{\bullet}_{\Gamma} \;:=\; \left\langle \frac{\bar\alpha_{I}}{\alpha_{I} }\, \bar x_{I} \;\middle\vert\; I \subseteq \{1,\ldots,n\} \text{ such that } \left.\left(\frac{\bar\alpha_{I}}{\alpha_{I}}\right)\right\lfloor_{\Gamma}=1 \right\rangle
\end{equation}
(where we shorten, e.g., $\alpha_I:=\alpha_{i_1}\cdot\cdots\cdot\alpha_{i_{k}}$ and $x_I := x_{i_1}\wedge\cdots\wedge x_{i_k}$ for a multi-index $I=\left(i_1,\dots,i_k\right)$ of length $|I|=k$).

The second author proved the following result.

\begin{theorem}[{\cite[Corollary 6.2 and its proof]{kasuya-holpar}}]\label{MMTT}
Let $G$ be a connected simply-connected complex solvable Lie group admitting a lattice $\Gamma$.
Consider the finite-dimensional sub-complex $B^{\bullet}_{\Gamma} \subset \left( \wedge^{0,\bullet} \solvmfd, \, \delbar \right)$ defined in \eqref{eq:def-b-holpar}.
Then the inclusion $B^{\bullet}_{\Gamma} \hookrightarrow \wedge^{0,\bullet}\solvmfd$ induces the cohomology isomorphism
\[H^{\bullet}\left(B^{\bullet}_{\Gamma},\, \delbar\right) \stackrel{\cong}{\to} H_{\bar\partial}^{0,\bullet}(\solvmfd) \;.\]
\end{theorem}

As regards Bott-Chern cohomology, define $\bar B^{\bullet}_{\Gamma} := \left\langle \frac{\alpha_{I}}{\bar\alpha_{I} } \, x_{I} \;\middle\vert\; I \subseteq \{1,\ldots,n\} \text{ such that } \left.\left(\frac{\alpha_{I}}{\bar\alpha_{I}}\right)\right\lfloor_{{\Gamma}}=1\right\rangle$, and
\begin{equation}\label{eq:def-c-cplx-par}
C^{\bullet_1,\bullet_2}_\Gamma \;:=\; \wedge^{\bullet_1} \g_{+}^{*} \otimes B^{\bullet_2}_{\Gamma} + \bar B^{\bullet_1}_{\Gamma} \otimes \wedge^{\bullet_2} \g_{-}^{*} \;.
\end{equation}
The authors proved the following result.

\begin{theorem}[{\cite[Theorem 2.24]{angella-kasuya-1}}]\label{palBCISO}
Let $G$ be a connected simply-connected complex solvable Lie group admitting a lattice $\Gamma$. Consider the finite-dimensional sub-double-complex $C^{\bullet,\bullet}_{\Gamma}\subset \wedge^{\bullet,\bullet}\solvmfd$ defined in \eqref{eq:def-c-cplx-par}.
Then the inclusion $C^{\bullet,\bullet}_{\Gamma} \hookrightarrow \wedge^{\bullet,\bullet}\solvmfd$ induces the cohomology isomorphism
\[ H^{\bullet,\bullet}_{BC}\left( C_{\Gamma}^{\bullet,\bullet} \right) \stackrel{\cong}{\to}  H^{\bullet,\bullet}_{BC}(\solvmfd) \;.\]
\end{theorem}

\medskip

Therefore, by Theorem \ref{thm:dolb-deformations} and Theorem \ref{thm:bc-deformations}, we get the following result, for which we provide explicit applications in the following.

\begin{corollary}
 Let $X$ be either a solvmanifold of splitting-type or a holomorphically parallelizable solvmanifold. Then the Dolbeault cohomology and the Bott-Chern cohomology both of $X$ and of some suitable small deformations of $X$ are computable by means of a finite-dimensional sub-double-complex of $\left( \wedge^{\bullet,\bullet}X,\, \del,\, \delbar\right)$.
\end{corollary}

\medskip

We note that small deformations of a holomorphically parallelizable solvmanifolds does not necessarily remain holomorphically parallelizable. This was firstly proved by I. Nakamura, providing explicit examples on the Iwasawa manifold, \cite[page 86, page 96]{nakamura}. In \cite{rollenske-jems}, S. Rollenske studied conditions for which a small deformation of a holomorphically parallelizable nilmanifold is still holomorphically parallelizable, \cite[Theorem 5.1]{rollenske-jems}, proving that non-tori holomorphically parallelizable nilmanifolds admit non-holomorphically parallelizable small deformations, \cite[Corollary 5.2]{rollenske-jems}.
We prove that the same holds true for holomorphically parallelizable solvmanifolds.
\begin{theorem}
Let $X=\solvmfd$ be a holomorphically parallelizable solvmanifold which is not a torus.
Then there exists a non-holomorphically parallelizable small deformation of $\solvmfd$.

\end{theorem}
\begin{proof}
By \cite[Corollary 5.2]{rollenske-jems}, we can assume that $\solvmfd$ is not a nilmanifold.
Take a connected simply-connected complex nilpotent subgroup $C\subset G$ such that $G = C \cdot N$, where $N$ is the nilradical of $G$.
We can take a  $1$-dimensional complex Lie subgroup $A\cong\bC$ with $A\subset C$ and a $1$-codimensional complex Lie subgroup $G^{\prime}$ with $N\subset G^{\prime}$ such that we have decomposition $G=A\ltimes G^{\prime}$.
Take a basis $\{x_{1},\dots,x_{n}\}$ of $\g_{+}^{*}$ which diagonalizes the semi-simple part of the $C$-action (where $\g_+$ denotes the Lie algebra of $G$-left-invariant holomorphic vector fields on $G$).
With respect to the above decomposition, we can take $x_{1}=\de z$ for a coordinate $z$ of the $1$-dimensional complex Lie subgroup $A$, and  $x_{2}=\esp^{a_{2}\,z}x^{\prime}_{2}$ for a non-trivial character $\esp^{a_{2}\,z}$ of $A$ and a holomorphic form $x^{\prime}_{j}$ on $G^{\prime}$, by trigonalizing the $A$-action.
Then the Dolbeault cohomology of $\left. \Gamma \middle\backslash G \right.$ is computed by means of 
\[ C^{\bullet,\bullet}_0 \;:=\; \wedge^{\bullet}\g_{+}^{*}\otimes B_{\Gamma}^{\bullet} \]
where
\[B_{\Gamma}^{\bullet} \;:=\;  \left\langle \frac{\bar\alpha_{I}}{\alpha_{I} }\, \bar x_{I} \;\middle\vert\; I \subseteq \{1,\ldots,n\} \text{ such that } \left.\left(\frac{\bar\alpha_{I}}{\alpha_{I}}\right)\right\lfloor_{\Gamma}=1 \right\rangle  \;,\]
(and where we shorten, e.g., $\alpha_I := \alpha_{i_1}\cdot\cdots\cdot\alpha_{i_k}$ and $x_I := x_{i_1}\wedge\cdots\wedge x_{i_k}$ for a multi-index $I=(i_1,\ldots,i_k)$ of length $|I|=k$).

We consider the family $\{J_{t}\}_t$ of deformations given by
$$ t\,\frac{\partial}{\partial z}\otimes \de \bar z \;\in\; H^{0,1}\left(X; T^{1,0}X\right) \;. $$

Then, for any $t$, we consider the double-complex
\[ D^{\bullet,\bullet}_t \;:=\; \wedge^{\bullet}\g_{+}^{*}(t)\otimes B_{\Gamma}^{\bullet}(t)\]
so that
\[\wedge^{\bullet}\g_{+}^{*}(t) \;=\; \wedge^{\bullet} \langle \de  z-t\,\de \bar z,\; x_{2},\;\dots,\; x_{n}\rangle\]
and 
\[B_{\Gamma}^{\bullet}(t) \;=\; \wedge\langle \de \bar z-\bar t\,\de  z\rangle\otimes  \left\langle \frac{\bar\alpha_{I}}{\alpha_{I} }\, \bar x^{\prime}_{I} \;\middle\vert\; I \subseteq \{2,\ldots,n\} \text{ such that } \left(\frac{\bar\alpha_{I}}{\alpha_{I}}\right)\lfloor_{\Gamma}=1 \right\rangle  \;,\]
and the $J_t$-Hermitian metrics
\[g_{t} \;:=\; \left(\de  z-t\,\de \bar z\right) \odot \left(\de \bar z-\bar t\,\de  z\right) + \sum_{j=2}^{n} x_{j}\odot\bar x_{j} \;.
\]
We can apply Theorem \ref{thm:dolb-deformations}.

Now we have
\[\bar\partial_{t} \left(\esp^{a_{2}\,z}x^{\prime}_{2}\right)\;=\;\frac{a_{2}\,t\,\left(\de \bar z-\bar t\,\de  z\right)}{1-\vert t\vert^{2}} \esp^{a_{2}\,z}x^{\prime}_{2} \;.\]
Hence we have, for $t\neq0$,
\[H^{1,0}_{\bar\partial_{t}}\left(\left. \Gamma \middle\backslash G \right.\right) \;=\; \ker \left. \bar\partial_{t} \right\lfloor_{\wedge^{1}\g_{+}^{*}(t)} \;\neq\; \wedge^{1}\g_{+}^{*}(t) \;.\]
By this, for $t\neq0$, we have $\dim_\bC H^{1,0}_{\bar\partial_{t}}\left(\left. \Gamma \middle\backslash G \right.\right) \;<\; \dim_\bC G$ and hence $(\left. \Gamma \middle\backslash G \right.,J_{t})$ is not holomorphically parallelizable.
\end{proof}

\section{Example: deformations of the Nakamura manifold}\label{naka}
 Consider the Lie group $G=\bC\ltimes_{\phi} \bC^{2}$ where \[\phi(z) \;=\; \left(
\begin{array}{cc}
\esp^{z}& 0  \\
0&    \esp^{-z}  
\end{array}
\right) \;. \]
Then there exists a lattice $\Gamma=(a\Z+2\pi\Z)\ltimes \Gamma^{\prime\prime}$ where $\Gamma^{\prime\prime}$ is a lattice in $\bC^{2}$.
The solvmanifold $X:=\left.\Gamma \middle\backslash G \right.$ is called \emph{(holomorphically parallelizable) Nakamura manifold}, \cite{nakamura}.

\medskip

In order to compute the Dolbeault, respectively Bott-Chern cohomologies of the Nakamura manifold, consider the sub-double-complexes $B^{\bullet,\bullet}_{\Gamma}$ and $C^{\bullet,\bullet}_{\Gamma}$ given in Table \ref{table:nak-b} and Table \ref{table:nak-c}, see \cite{kasuya-mathz, angella-kasuya-1}. (For the sake of simplicity, we shorten, e.g., $\de z_{2\bar3}:=\de z_{2}\wedge\de \bar z_3$, where $z_1$ is the holomorphic coordinate on $\bC$ and $\{z_2, z_3\}$ is the set of holomorphic coordinates on $\bC^2$.)

\begin{table}[!hpt]\label{table:nak-b}
\centering
{\resizebox{\textwidth}{!}{
\begin{tabular}{>{$\mathbf\bgroup}l<{\mathbf\egroup$} || >{$}l<{$}}
\toprule
 & B^{\bullet,\bullet}_{\Gamma}\\
\toprule
(0,0) & \bC \left\langle 1 \right\rangle \\
\midrule[0.02em]
(1,0) & \bC \left\langle \de z_{1},\; \esp^{-z_{1}}\de z_{2},\; \esp^{z_{1}}\de z_{3} \right\rangle  \\[5pt]
(0,1) & \bC \left\langle \de z_{\bar1},\; \esp^{-z_{1}}\de z_{\bar2},\; \esp^{z_{1}}\de z_{\bar3}\right\rangle \\
\midrule[0.02em]
(2,0) & \bC \left\langle \esp^{-z_{1}}\de z_{12},\; \esp^{z_{1}}\de z_{13},\; \de z_{23} \right\rangle \\[5pt]
(1,1) & \bC \left\langle \de z_{1\bar1},\; \esp^{-z_{1}}\de z_{1\bar2},\; \esp^{z_{1}}\de z_{1\bar3},\; \esp^{-z_{1}}\de z_{2\bar1},\; \esp^{-2z_{1}}\de z_{2\bar2},\; \de z_{2\bar3},\; \esp^{z_{1}}\de z_{3\bar1},\; \de z_{3\bar2},\; \esp^{2z_{1}}\de z_{3\bar3} \right\rangle \\[5pt]
(0,2) & \bC \left\langle \esp^{-z_{1}} \de z_{\bar1\bar2},\; \esp^{z_{1}} \de z_{\bar1\bar3},\; \de z_{\bar 2\bar3} \right\rangle \\
\midrule[0.02em]
(3,0) & \bC \left\langle \de z_{123} \right\rangle \\[5pt]
(2,1) & \bC \left\langle \esp^{-z_{1}}\de z_{12\bar1},\; \esp^{-2 z_{1}}\de z_{12\bar2},\; \de z_{12\bar3},\; \esp^{z_{1}}\de z_{13\bar1},\; \de z_{13\bar2},\; \esp^{2z_{1}}\de z_{13\bar3},\; \de z_{23\bar1},\; \esp^{-z_{1}}\de z_{23\bar2},\; \esp^{z_{1}}\de z_{23\bar3}\right\rangle \\[5pt]
(1,2) & \bC \left\langle \de z_{3\bar1\bar2},\;  \de z_{2\bar1\bar3},\;  \de z_{1\bar2\bar3},\; \esp^{ -z_{1}}\de z_{1\bar1\bar2},\; \esp^{z_{1}}\de z_{1\bar1\bar3},\; \esp^{-2 z_{1}}\de z_{2\bar1\bar2},\; \esp^{-z_{1}}\de z_{2\bar2\bar3},\; \esp^{2 z_{1}}\de z_{3\bar1\bar3},\; \esp^{ z_{1}}\de z_{3\bar2\bar3} \right\rangle \\[5pt]
(0,3) & \bC \left\langle \de z_{\bar1\bar2\bar3} \right\rangle \\
\midrule[0.02em]
(3,1) & \bC \left\langle \de z_{123\bar1},\; \esp^{-z_{1}}\de z_{123\bar2},\; \esp^{z_{1}}\de z_{123\bar3} \right\rangle \\[5pt]
(2,2) & \bC \left\langle \esp^{-2z_{1}}\de z_{12\bar1\bar2},\; \de z_{12\bar1\bar3},\; \esp^{-z_{1}}\de z_{12\bar2\bar3},\; \de z_{13\bar1\bar2},\; \esp^{2z_{1}}\de z_{13\bar1\bar3},\; \esp^{z_{1}}\de z_{13\bar2\bar3},\; \esp^{-z_{1}}\de z_{23\bar1\bar2},\; \esp^{z_{1}}\de z_{23\bar1\bar3},\; \de z_{23\bar2\bar3} \right\rangle \\[5pt]
(1,3) & \bC \left\langle \de z_{1\bar1\bar2\bar3},\; \esp^{- z_{1}}\de z_{2\bar1\bar2\bar3},\; \esp^{ z_{1}}\de z_{3\bar1\bar2\bar3} \right\rangle \\
\midrule[0.02em]
(3,2) & \bC \left\langle \esp^{-z_{1}}\de z_{123\bar1\bar2},\; \esp^{z_{1}}\de z_{123\bar1\bar3},\; \de z_{123\bar2\bar3} \right\rangle \\[5pt]
(2,3) & \bC \left\langle \esp^{-z_{1}}\de z_{12\bar1\bar2\bar3},\; \esp^{z_{1}}\de z_{13\bar1\bar2\bar3},\; \de z_{23\bar1\bar2\bar3} \right\rangle \\
\midrule[0.02em]
(3,3) & \bC \left\langle \de z_{123\bar1\bar2\bar3} \right\rangle \\
\bottomrule
\end{tabular}
}}
\caption{The double-complex $B^{\bullet,\bullet}_{\Gamma}$ for computing the Dolbeault cohomology of the holomorphically parallelizable Nakamura manifold $\left. \Gamma \middle\backslash G \right.$.}
\end{table}

\begin{table}[!hpt]\label{table:nak-c}
\centering
{\resizebox{\textwidth}{!}{
\begin{tabular}{>{$\mathbf\bgroup}l<{\mathbf\egroup$} || >{$}l<{$}}
\toprule
 & C^{\bullet,\bullet}_{\Gamma}\\
\toprule
(0,0) & \bC \left\langle 1 \right\rangle \\
\midrule[0.02em]
(1,0) & \bC \left\langle \de z_{1},\; \esp^{-z_{1}}\de z_{2},\; \esp^{z_{1}}\de z_{3},\; \esp^{-\bar z_{1}}\de z_{2},\; \esp^{\bar z_{1}}\de z_{3} \right\rangle  \\[5pt]
(0,1) & \bC \left\langle \de z_{\bar1},\; \esp^{-z_{1}}\de z_{\bar2},\; \esp^{z_{1}}\de z_{\bar3},\; \esp^{-\bar z_{1}}\de z_{\bar2},\; \esp^{\bar z_{1}}\de z_{\bar3} \right\rangle \\
\midrule[0.02em]
(2,0) & \bC \left\langle \esp^{-z_{1}}\de z_{12},\; \esp^{z_{1}}\de z_{13},\; \de z_{23},\; \esp^{-\bar z_{1}}\de z_{12},\; \esp^{\bar z_{1}}\de z_{13} \right\rangle \\[5pt]
(1,1) & \bC \left\langle \de z_{1\bar1},\; \esp^{-z_{1}}\de z_{1\bar2},\; \esp^{z_{1}}\de z_{1\bar3},\; \esp^{-z_{1}}\de z_{2\bar1},\; \esp^{-2z_{1}}\de z_{2\bar2},\; \de z_{2\bar3},\; \esp^{z_{1}}\de z_{3\bar1},\; \de z_{3\bar2},\; \esp^{2z_{1}}\de z_{3\bar3}, \right. \\[5pt]
& \left.\esp^{-\bar z_{1}}\de z_{2\bar1},\; \esp^{-\bar z_{1}}\de z_{1\bar2},\; \esp^{\bar z_{1}}\de z_{1\bar3},\; \esp^{\bar z_{1}}\de z_{3\bar1},\; \esp^{-2\bar z_{1}} \de z_{2\bar2},\; \esp^{2\bar z_{1}}\de z_{3\bar3} \right\rangle \\[5pt]
(0,2) & \bC \left\langle \esp^{-z_{1}} \de z_{\bar1\bar2},\; \esp^{z_{1}} \de z_{\bar1\bar3},\; \de z_{\bar 2\bar3},\; \esp^{-\bar z_{1}}\de z_{\bar1\bar2},\; \esp^{\bar z_{1}}\de z_{\bar1\bar3} \right\rangle \\
\midrule[0.02em]
(3,0) & \bC \left\langle \de z_{123} \right\rangle \\[5pt]
(2,1) & \bC \left\langle \esp^{-z_{1}}\de z_{12\bar1},\; \esp^{-2 z_{1}}\de z_{12\bar2},\; \de z_{12\bar3},\; \esp^{z_{1}}\de z_{13\bar1},\; \de z_{13\bar2},\; \esp^{2z_{1}}\de z_{13\bar3},\; \de z_{23\bar1},\; \esp^{-z_{1}}\de z_{23\bar2},\; \esp^{z_{1}}\de z_{23\bar3}, \right. \\[5pt]
& \left. \esp^{-\bar z_{1}}\de z_{12\bar1},\; \esp^{\bar z_{1}}\de z_{13\bar1},\; \esp^{-2\bar z_{1}}\de z_{12\bar2},\; \esp^{-\bar z_{1}}\de z_{23\bar2},\; \esp^{2\bar z_{1}}\de z_{13\bar3},\; \esp^{\bar z_{1}}\de z_{23\bar3} \right\rangle \\[5pt]
(1,2) & \bC \left\langle \esp^{-\bar z_{1}}\de z_{1\bar1\bar2},\; \esp^{-2\bar z_{1}}\de z_{2\bar1\bar2},\; \de z_{3\bar1\bar2},\; \esp^{\bar z_{1}}\de z_{1\bar1\bar3},\; \de z_{2\bar1\bar3},\; \esp^{2\bar z_{1}}\de z_{3\bar1\bar3},\; \de z_{1\bar2\bar3},\; \esp^{-\bar z_{1}}\de z_{2\bar2\bar3},\; \esp^{\bar z_{1}}\de z_{3\bar2\bar3}, \right. \\[5pt]
& \left. \esp^{ -z_{1}}\de z_{1\bar1\bar2},\; \esp^{z_{1}}\de z_{1\bar1\bar3},\; \esp^{-2 z_{1}}\de z_{2\bar1\bar2},\; \esp^{-z_{1}}\de z_{2\bar2\bar3},\; \esp^{2 z_{1}}\de z_{3\bar1\bar3},\; \esp^{ z_{1}}\de z_{3\bar2\bar3} \right\rangle \\[5pt]
(0,3) & \bC \left\langle \de z_{\bar1\bar2\bar3} \right\rangle \\
\midrule[0.02em]
(3,1) & \bC \left\langle \de z_{123\bar1},\; \esp^{-z_{1}}\de z_{123\bar2},\; \esp^{z_{1}}\de z_{123\bar3},\; \esp^{-\bar z_{1}}\de z_{123\bar2},\; \esp^{\bar z_{1}}\de z_{123\bar3} \right\rangle \\[5pt]
(2,2) & \bC \left\langle \esp^{-2z_{1}}\de z_{12\bar1\bar2},\; \de z_{12\bar1\bar3},\; \esp^{-z_{1}}\de z_{12\bar2\bar3},\; \de z_{13\bar1\bar2},\; \esp^{2z_{1}}\de z_{13\bar1\bar3},\; \esp^{z_{1}}\de z_{13\bar2\bar3},\; \esp^{-z_{1}}\de z_{23\bar1\bar2},\; \esp^{z_{1}}\de z_{23\bar1\bar3}, \right.\\[5pt]
& \left. \de z_{23\bar2\bar3}, \; \esp^{-2\bar z_{1}}\de z_{12\bar1\bar2},\; \esp^{-\bar z_{1}}\de z_{23\bar1\bar2},\; \esp^{-\bar z_{1}}\de z_{12\bar2\bar3},\; \esp^{\bar z_{1}}\de z_{13\bar2\bar3},\; \esp^{2\bar z_{1}}\de z_{13\bar1\bar3},\; \esp^{\bar z_{1}}\de z_{23\bar1\bar3} \right\rangle \\[5pt]
(1,3) & \bC \left\langle \de z_{1\bar1\bar2\bar3},\; \esp^{-\bar z_{1}}\de z_{2\bar1\bar2\bar3},\; \esp^{\bar z_{1}}\de z_{3\bar1\bar2\bar3},\; \esp^{- z_{1}}\de z_{2\bar1\bar2\bar3},\; \esp^{ z_{1}}\de z_{3\bar1\bar2\bar3} \right\rangle \\
\midrule[0.02em]
(3,2) & \bC \left\langle \esp^{-z_{1}}\de z_{123\bar1\bar2},\; \esp^{z_{1}}\de z_{123\bar1\bar3},\; \de z_{123\bar2\bar3},\; \esp^{-\bar z_{1}}\de z_{123\bar1\bar2},\; \esp^{\bar z_{1}}\de z_{123\bar1\bar3} \right\rangle \\[5pt]
(2,3) & \bC \left\langle \esp^{-z_{1}}\de z_{12\bar1\bar2\bar3},\; \esp^{z_{1}}\de z_{13\bar1\bar2\bar3},\; \de z_{23\bar1\bar2\bar3},\; 
\esp^{-\bar z_{1}}\de z_{12\bar1\bar2\bar3},\; \esp^{\bar z_{1}}\de z_{13\bar1\bar2\bar3} \right\rangle \\
\midrule[0.02em]
(3,3) & \bC \left\langle \de z_{123\bar1\bar2\bar3} \right\rangle \\
\bottomrule
\end{tabular}
}}
\caption{The double-complex $C^{\bullet,\bullet}_\Gamma$ for computing the Bott-Chern cohomology of the holomorphically parallelizable Nakamura manifold $\left. \Gamma \middle\backslash G \right.$.}
\end{table}

Then, by \cite[Corollary 6.2]{kasuya-holpar} and by \cite[Theorem 2.24]{angella-kasuya-1}, the inclusions $B^{\bullet,\bullet}_{\Gamma}\subset \wedge^{\bullet,\bullet}(X)$ and $C^{\bullet,\bullet}_{\Gamma}\subset \wedge^{\bullet,\bullet}(X)$
induce isomorphisms
\[ H^{\bullet,\bullet}_{\bar\partial}(B^{\bullet,\bullet}_{\Gamma})\cong H^{\bullet,\bullet}_{\bar\partial}(X) \qquad \text{ and }\qquad H^{\bullet,\bullet}_{BC}(C^{\bullet,\bullet}_{\Gamma})\cong H^{\bullet,\bullet}_{BC}(X) \;. \]

\medskip

We consider deformations $\{J_{t}\}_{t\in B}$ over a ball $B\subset \bC$ given by
\begin{enumerate}[{\itshape (1)}]
 \item $t\,\frac{\partial}{\partial z_{1}}\otimes \de \bar z_{1} \in H^{0,1}\left(X; T^{1,0}X\right)$, or
 \item $t\,\frac{\partial}{\partial z_{1}}\otimes \esp^{ z_{1}}\de \bar z_{3} \in H^{0,1}\left(X; T^{1,0}X\right)$.
\end{enumerate}

\medskip

As for deformations in case {\itshape (1)}, we can compute the Dolbeault and Bott-Chern cohomologies by applying Theorem \ref{thm:dolb-deformations} and Theorem \ref{thm:bc-deformations} to the complexes $B^{\bullet,\bullet}_\Gamma(t)$ and $C^{\bullet,\bullet}_\Gamma(t)$ in Table \ref{table:nak-def-b} and Table \ref{table:nak-def-c}, respectively, and by considering the $J_{t}$-Hermitian metrics $g_{t}:=\phi^{1,0}_{1}(t)\odot\phi^{0,1}_{1}(t)+\phi^{1,0}_{2}(t)\odot\varphi^{0,1}_{2}(t)+\phi^{1,0}_{3}(t)\odot\varphi^{0,1}_{3}(t)$; the generators of the complexes are defined starting from the forms in Table \ref{table:nak-def-1-c-gen}, and we summarize the results of the computations of the Dolbeault and Bott-Chern cohomologies in Table \ref{table:nak-def-1-delbar} and Table \ref{table:nak-def-1-bc}, respectively.

As for deformations in case {\itshape (2)}, we can compute the Dolbeault cohomology by applying Theorem \ref{thm:dolb-deformations} to the complex $B^{\bullet,\bullet}_\Gamma(t)$ in Table \ref{table:nak-def-b}, and by considering the $J_{t}$-Hermitian metrics $g_{t}:=\phi^{1,0}_{1}(t)\odot\phi^{0,1}_{1}(t)+\phi^{1,0}_{2}(t)\odot\varphi^{0,1}_{2}(t)+\phi^{1,0}_{3}(t)\odot\varphi^{0,1}_{3}(t)$; the generators of the complex are defined starting from the forms in Table \ref{table:nak-def-2-c-gen}, and we summarize the results of the computation of the Dolbeault cohomology in Table \ref{table:nak-def-2-delbar}. (As regards the Bott-Chern cohomology for deformations in case {\itshape (2)}, the vector space $C^{\bullet,\bullet}_\Gamma(t)$ does not provide a sub-double-complex for $t\neq0$, and, by modifying it in order to be closed for both $\del_t$ and $\delbar_t$, and $\bar*_{g_t}$, as required in Theorem \ref{thm:bc-deformations}, it seems that the finite-dimensionality is no more guaranteed.)

\begin{remark}{\rm
In \cite[Theorem 4]{Hd}, K. Hasegawa showed that deformations in case {\itshape (2)} are not left-invariant.
Hence our method is effective for computing the Dolbeault cohomology of non-left-invariant complex structures.
}\end{remark}

(As a matter of notations, we shorten, e.g., $\phi^{1,0}_1(t)\wedge\phi^{0,1}_{12}(t):=\phi^{1,0}_1(t)\wedge\phi^{0,1}_{1}(t)\wedge\phi^{0,1}_{2}(t)$.)

\begin{table}[!hpt]\label{table:nak-def-1-c-gen}
 \centering
\begin{tabular}{>{$}l<{$} || >{$}l<{$}}
\toprule
\text{\bfseries case {\itshape (1)}} & \\
\mathbf{\psi} & \mathbf{\de\psi} \\
\toprule
\phi^{1,0}_{1}(t):=\de z_{1}-t\de \bar z_{1} & \de \phi^{1,0}_{1}(t) = 0 \\[5pt]
\phi^{1,0}_{2}(t):=\esp^{- z_{1}}\de z_{2} & \de \phi^{1,0}_{2}(t) = -\frac{1}{1-\vert t\vert^{2}}\, \phi^{1,0}_{1}(t)\wedge \phi^{1,0}_{2}(t)+\frac{t}{1-\vert t\vert^{2}}\, \phi^{1,0}_{2}(t) \wedge \phi^{0,1}_{1}(t) \\[5pt]
\phi^{1,0}_{3}(t):=\esp^{ z_{1}}\de z_{3} & \de \phi^{1,0}_{3}(t) = \frac{1}{1-\vert t\vert^{2}}\,\phi^{1,0}_{1}(t)\wedge \phi^{1,0}_{3}(t)-\frac{t}{1-\vert t\vert^{2}}\,\phi^{1,0}_{3}(t) \wedge \phi^{0,1}_{1}(t) \\
\midrule[0.02em]
\varphi^{1,0}_{2}(t):=\esp^{-\bar z_{1}}\de z_{2} & \de \varphi^{1,0}_{2}(t) = -\frac{\bar t}{1-\vert t\vert^{2}}\,\phi^{1,0}_{1}(t) \wedge \varphi^{1,0}_{2}(t)+\frac{1}{1-\vert t\vert^{2}}\, \varphi^{1,0}_{2}(t) \wedge \phi^{0,1}_{1}(t) \\[5pt]
\varphi^{1,0}_{3}(t):=\esp^{\bar z_{1}}\de z_{3} & \de \varphi^{1,0}_{3}(t) = \frac{\bar t}{1-\vert t\vert^{2}}\,\phi^{1,0}_{1}(t)\wedge \varphi^{1,0}_{3}(t)-\frac{1}{1-\vert t\vert^{2}}\varphi^{1,0}_{3}(t)\wedge\phi^{0,1}_{1}(t) \\
\midrule[0.02em] \midrule[0.02em]
\phi^{0,1}_{1}(t):=\de\bar z_{1}-\bar t\de  z_{1} & \de \phi^{0,1}_{1}(t) = 0 \\[5pt]
\phi^{0,1}_{2}(t):=\esp^{- z_{1}}\de \bar z_{2} & \de \phi^{0,1}_{2}(t) = -\frac{1}{1-\vert t\vert^{2}}\phi^{1,0}_{1}(t)\wedge \phi^{0,1}_{2}(t)-\frac{t}{1-\vert t\vert^{2}}\,\phi^{0,1}_{1}(t)\wedge \phi^{0,1}_{2}(t) \\[5pt]
\phi^{0,1}_{3}(t):=\esp^{z_{1}}\de \bar z_{3} & \de \phi^{0,1}_{3}(t) = \frac{1}{1-\vert t\vert^{2}}\, \phi^{1,0}_{1}(t)\wedge \phi^{0,1}_{3}(t)+\frac{t}{1-\vert t\vert^{2}}\, \phi^{0,1}_{1}(t)\wedge \phi^{0,1}_{3}(t) \\
\midrule[0.02em]
\varphi^{0,1}_{2}(t):=\esp^{-\bar z_{1}}\de \bar z_{2} & \de \varphi^{0,1}_{2}(t) = -\frac{\bar t}{1-\vert t\vert^{2}}\, \phi^{1,0}_{1}(t)\wedge \varphi^{0,1}_{2}(t)-\frac{1}{1-\vert t\vert^{2}}\, \phi^{0,1}_{1}(t)\wedge \varphi^{0,1}_{2}(t) \\[5pt]
\varphi^{0,1}_{3}(t):=\esp^{\bar z_{1}}\de\bar z_{3} & \de \varphi^{0,1}_{3}(t) = \frac{\bar t}{1-\vert t\vert^{2}}\, \phi^{1,0}_{1}(t)\wedge \varphi^{0,1}_{3}(t)+\frac{1}{1-\vert t\vert^{2}}\, \phi^{0,1}_{1}(t)\wedge \varphi^{0,1}_{3}(t) \\
\bottomrule
\end{tabular}
\caption{Definitions for setting the generators of the complexes $B^{\bullet,\bullet}_\Gamma(t)$, see Table \ref{table:nak-def-b}, and $C^{\bullet,\bullet}_\Gamma(t)$, see Table \ref{table:nak-def-c}, for the deformations in case {\itshape (1)}, which are given by $t\,\frac{\partial}{\partial z_{1}}\otimes \de \bar z_{1}$, of the holomorphically parallelizable Nakamura manifold $\left. \Gamma \middle\backslash G \right.$.}
\end{table}

\begin{table}[!hpt]\label{table:nak-def-2-c-gen}
 \centering
\begin{tabular}{>{$}l<{$} || >{$}l<{$}}
\toprule
\text{\bfseries case {\itshape (2)}} & \\
\mathbf{\psi} & \mathbf{\de\psi} \\
\toprule
\phi^{1,0}_{1}(t):=\de z_{1}-t\esp^{ z_{1}}\de\bar z_{3} & \de \phi^{1,0}_{1}(t)=-t\, \phi^{1,0}_{1}(t)\wedge \phi^{0,1}_{3}(t) \\[5pt]
\phi^{1,0}_{2}(t):=\esp^{- z_{1}}\de z_{2} & \de \phi^{1,0}_{2}(t)=-\phi^{1,0}_{1}(t)\wedge \phi^{1,0}_{2}(t)+t\, \phi^{1,0}_{2}(t)\wedge \phi^{0,1}_{3}(t) \\[5pt]
\phi^{1,0}_{3}(t):=\esp^{ z_{1}}\de z_{3} & \de \phi^{1,0}_{3}(t)=\phi^{1,0}_{1}(t)\wedge \phi^{1,0}_{3}(t)-t\, \phi^{1,0}_{3}(t)\wedge \phi^{0,1}_{3}(t) \\
\midrule[0.02em]
\varphi^{1,0}_{2}(t):=\esp^{-\bar z_{1}}\de z_{2} & \de \varphi^{1,0}_{2}(t)=\bar t\, \varphi^{1,0}_{2}(t)\wedge \varphi^{1,0}_{3}(t)+\varphi^{1,0}_{2}(t)\wedge \phi^{0,1}_{1}(t) \\[5pt]
\varphi^{1,0}_{3}(t):=\esp^{\bar z_{1}}\de z_{3} & \de \varphi^{1,0}_{3}(t)=-\varphi^{1,0}_{3}(t)\wedge \phi^{0,1}_{1}(t) \\
\midrule[0.02em] \midrule[0.02em]
\phi^{0,1}_{1}(t):=\de\bar z_{1}-\bar t\esp^{\bar z_{1}}\de z_{3} & \de \phi^{0,1}_{1}(t)= \bar t\, \varphi^{1,0}_{3}(t)\wedge\phi^{0,1}_{1}(t) \\[5pt]
\phi^{0,1}_{2}(t):=\esp^{-z_{1}}\de\bar z_{2} & \de \phi^{0,1}_{2}(t)=-\phi^{1,0}_{1}(t)\wedge \phi^{0,1}_{2}(t)+t\, \phi^{0,1}_{2}(t)\wedge \phi^{0,1}_{3}(t) \\[5pt]
\phi^{0,1}_{3}(t):=\esp^{z_{1}}\de\bar z_{3} & \de \phi^{0,1}_{3}(t)=\phi^{1,0}_{1}(t)\wedge \phi^{0,1}_{3}(t) \\
\midrule[0.02em]
\varphi^{0,1}_{2}(t):=\esp^{-\bar z_{1}}\de \bar z_{2} & \de \varphi^{0,1}_{2}(t)=-\bar t\,\varphi^{1,0}_{3}(t)\wedge \varphi^{0,1}_{2}(t)-\phi^{0,1}_{1}(t)\wedge \varphi^{0,1}_{2}(t) \\[5pt]
\varphi^{0,1}_{3}(t):=\esp^{\bar z_{1}}\de\bar z_{3} & \de \varphi^{0,1}_{3}(t)=\bar t\,\varphi^{1,0}_{3}(t)\wedge \varphi^{0,1}_{3}(t)+\phi^{0,1}_{1}(t)\wedge \varphi^{0,1}_{3}(t) \\
\bottomrule
\end{tabular}
\caption{Definitions for setting the generators of the complex $B^{\bullet,\bullet}_\Gamma(t)$, see Table \ref{table:nak-def-b}, for the deformations in case {\itshape (2)}, which are given by $t\,\frac{\partial}{\partial z_{1}}\otimes \esp^{ z_{1}}\de \bar z_{3}$, of the holomorphically parallelizable Nakamura manifold $\left. \Gamma \middle\backslash G \right.$.}
\end{table}

\begin{table}[!hpt]\label{table:nak-def-b}
 \centering
\begin{tabular}{>{$\mathbf\bgroup}l<{\mathbf\egroup$} || >{$}l<{$}}
\toprule
 & B^{\bullet,\bullet}_{\Gamma}(t)\\
\toprule
(0,0) & \bC \left\langle 1 \right\rangle \\
\midrule[0.02em]
(1,0) & \bC \left\langle \phi^{1,0}_{1}(t),\; \phi^{1,0}_{2}(t),\; \phi^{1,0}_{3} (t)\right\rangle  \\[5pt]
(0,1) & \bC \left\langle \phi^{0,1}_{1}(t),\; \phi^{0,1}_{2}(t),\; \phi^{0,1}_{3}(t)\right\rangle \\
\midrule[0.02em]
(2,0) & \bC \left\langle \phi^{1,0}_{12}(t),\; \phi^{1,0}_{13}(t),\; \phi^{1,0}_{23}(t) \right\rangle \\[5pt]
(1,1) & \bC \left\langle \phi^{1,0}_{1}(t)\wedge\phi^{0,1}_{1}(t),\; \phi^{1,0}_{1}(t)\wedge\phi^{0,1}_{2}(t),\; \phi^{1,0}_{1}(t)\wedge\phi^{0,1}_{3}(t),\; \phi^{1,0}_{2}(t)\wedge\phi^{0,1}_{1}(t),\; \phi^{1,0}_{2}(t)\wedge\phi^{0,1}_{2}(t),\;\right.\\[5pt]
& \left. \phi^{1,0}_{2}(t)\wedge\phi^{0,1}_{3}(t),\; \phi^{1,0}_{3}(t)\wedge\phi^{0,1}_{1}(t),\; \phi^{1,0}_{3}(t)\wedge\phi^{0,1}_{2}(t),\; \phi^{1,0}_{3}(t)\wedge\phi^{0,1}_{3} (t)\right\rangle \\[5pt]
(0,2) & \bC \left\langle \phi^{0,1}_{12}(t),\; \phi^{0,1}_{13}(t),\; \phi^{0,1}_{ 23}(t) \right\rangle \\
\midrule[0.02em]
(3,0) & \bC \left\langle \phi^{1,0}_{123}(t) \right\rangle \\[5pt]
(2,1) & \bC \left\langle \phi^{1,0}_{12}(t)\wedge\phi^{0,1}_{1}(t),\; \phi^{1,0}_{12}(t)\wedge\phi^{0,1}_{2}(t),\; \phi^{1,0}_{12}(t)\wedge\phi^{0,1}_{3}(t),\; \phi^{1,0}_{13}(t)\wedge\phi^{0,1}_{1}(t),\; \phi^{1,0}_{13}(t)\wedge \phi^{0,1}_{2}(t),\; \right.\\[5pt]
& \left. \phi^{1,0}_{13}(t)\wedge\phi^{0,1}_{3}(t),\; \phi^{1,0}_{23}(t)\wedge\phi^{0,1}_{1}(t),\; \phi^{1,0}_{23}(t)\wedge \phi^{0,1}_{2}(t),\; \phi^{1,0}_{23}(t)\wedge\phi^{0,1}_{3}(t)\right\rangle \\[5pt]
(1,2) & \bC \left\langle \phi^{1,0}_{3}(t)\wedge\phi^{0,1}_{12}(t),\;  \phi^{1,0}_{2}(t)\wedge\phi^{0,1}_{13}(t),\;  \phi^{1,0}_{1}(t)\wedge\phi^{0,1}_{23}(t),\; \phi^{1,0}_{1}(t)\wedge\phi^{0,1}_{12}(t),\; \phi^{1,0}_{1}(t)\wedge\phi^{0,1}_{13}(t),\; 
\; \right.\\[5pt]
& \left. \phi^{1,0}_{2}(t)\wedge\phi^{0,1}_{12}(t),\; \phi^{1,0}_{2}(t)\wedge\phi^{0,1}_{23}(t),\; \phi^{1,0}_{3}(t)\wedge\phi^{0,1}_{13}(t),\; \phi^{1,0}_{3}(t)\wedge\phi^{0,1}_{23}(t) \right\rangle \\[5pt]
(0,3) & \bC \left\langle \phi^{0,1}_{123} (t)\right\rangle \\
\midrule[0.02em]
(3,1) & \bC \left\langle \phi^{1,0}_{123}(t)\wedge\phi^{0,1}_{1}(t),\; \phi^{1,0}_{123}(t)\wedge\phi^{0,1}_{2}(t),\; \phi^{1,0}_{123}(t)\wedge\phi^{0,1}_{3}(t) \right\rangle \\[5pt]
(2,2) & \bC \left\langle \phi^{1,0}_{12}(t)\wedge\phi^{0,1}_{12}(t),\; \phi^{1,0}_{12}(t)\wedge\phi^{0,1}_{13}(t),\; \phi^{1,0}_{12}(t)\wedge\phi^{0,1}_{23}(t),\; \phi^{1,0}_{13}(t)\wedge\phi^{0,1}_{12}(t),\; \phi^{1,0}_{13}(t)\wedge\phi^{0,1}_{13}(t),\;  \right.\\[5pt]
& \left.
\phi^{1,0}_{13}(t)\wedge\phi^{0,1}_{23}(t) ,\; \phi^{1,0}_{23}(t)\wedge\phi^{0,1}_{12}(t),\; \phi^{1,0}_{23}(t)\wedge\phi^{0,1}_{13}(t), \phi^{1,0}_{23}(t)\wedge\phi^{0,1}_{23}(t) \right\rangle \\[5pt]
(1,3) & \bC \left\langle \phi^{1,0}_{1}(t)\wedge\phi^{0,1}_{123}(t),\; \phi^{1,0}_{2}(t)\wedge\phi^{0,1}_{123}(t),\;\phi^{1,0}_{3}(t)\wedge\phi^{0,1}_{123}(t) \right\rangle \\
\midrule[0.02em]
(3,2) & \bC \left\langle \phi^{1,0}_{123}(t)\wedge\phi^{0,1}_{12}(t),\; \phi^{1,0}_{123}(t)\wedge\phi^{0,1}_{13}(t),\; \phi^{1,0}_{123}(t)\wedge\phi^{0,1}_{23}(t) \right\rangle \\[5pt]
(2,3) & \bC \left\langle \phi^{1,0}_{12}(t)\wedge \phi^{0,1}_{123}(t),\; \phi^{1,0}_{13}(t)\wedge \phi^{0,1}_{123}(t),\; \phi^{1,0}_{23}(t)\wedge \phi^{0,1}_{123}(t) \right\rangle \\
\midrule[0.02em]
(3,3) & \bC \left\langle\phi^{1,0}_{123}(t)\wedge \phi^{0,1}_{123}(t) \right\rangle \\
\bottomrule
\end{tabular}
\caption{The double-complex $B^{\bullet,\bullet}_{\Gamma}(t)$ for computing the Dolbeault cohomology of the small deformations in case {\itshape (1)} and in case {\itshape (2)} of the holomorphically parallelizable Nakamura manifold $\left. \Gamma \middle\backslash G \right.$.}
\end{table}

\begin{table}[!hpt]\label{table:nak-def-c}
 \centering
{\resizebox{\textwidth}{!}{
\begin{tabular}{>{$\mathbf\bgroup}l<{\mathbf\egroup$} || >{$}l<{$}}
\toprule
 & C^{\bullet,\bullet}_{\Gamma}(t)\\
\toprule
(0,0) & \bC \left\langle 1 \right\rangle \\
\midrule[0.02em]
(1,0) & \bC \left\langle \phi^{1,0}_{1}(t),\; \phi^{1,0}_{2}(t),\; \phi^{1,0}_{3} (t),\; \varphi^{1,0}_{2}(t),\; \varphi^{1,0}_{3} (t)\right\rangle  \\[5pt]
(0,1) & \bC \left\langle \phi^{0,1}_{1}(t),\; \phi^{0,1}_{2}(t),\; \phi^{0,1}_{3}(t),\; \varphi^{0,1}_{2}(t),\; \varphi^{0,1}_{3}(t)\right\rangle \\
\midrule[0.02em]
(2,0) & \bC \left\langle \phi^{1,0}_{12}(t),\; \phi^{1,0}_{13}(t),\; \phi^{1,0}_{23}(t),\; \phi^{1,0}_{1}(t)\wedge\varphi^{1,0}_{2}(t),\; \phi^{1,0}_{1}(t)\wedge\varphi^{1,0}_{3}(t) \right\rangle \\[5pt]
(1,1) & \bC \left\langle \phi^{1,0}_{1}(t)\wedge\phi^{0,1}_{1}(t),\; \phi^{1,0}_{1}(t)\wedge\phi^{0,1}_{2}(t),\; \phi^{1,0}_{1}(t)\wedge\phi^{0,1}_{3}(t),\; \phi^{1,0}_{2}(t)\wedge\phi^{0,1}_{1}(t),\; \phi^{1,0}_{2}(t)\wedge\phi^{0,1}_{2}(t),\;\right.\\[5pt]
& \left. \phi^{1,0}_{2}(t)\wedge\phi^{0,1}_{3}(t),\; \phi^{1,0}_{3}(t)\wedge\phi^{0,1}_{1}(t),\; \phi^{1,0}_{3}(t)\wedge\phi^{0,1}_{2}(t),\; \phi^{1,0}_{3}(t)\wedge\phi^{0,1}_{3}(t),\; \phi_{1}^{1,0}(t)\wedge\varphi^{0,1}_{2}(t),\;
\right. \\[5pt]
& \left. \phi_{1}^{1,0}(t)\wedge\varphi^{0,1}_{3}(t),\; \varphi_{2}^{1,0}(t)\wedge\phi^{0,1}_{1}(t),\; \varphi_{2}^{1,0}(t)\wedge\varphi^{0,1}_{2}(t),\; \varphi_{3}^{1,0}(t)\wedge\phi^{0,1}_{1}(t), \varphi_{3}^{1,0}(t)\wedge\varphi^{0,1}_{3}(t)
\right\rangle \\[5pt]
(0,2) & \bC \left\langle \phi^{0,1}_{12}(t),\; \phi^{0,1}_{13}(t),\; \phi^{0,1}_{ 23}(t),\;  \phi^{0,1}_{1}(t)\wedge\varphi^{0,1}_{2}(t),\; \phi^{0,1}_{1}(t)\wedge\varphi^{0,1}_{3}(t) \right\rangle \\
\midrule[0.02em]
(3,0) & \bC \left\langle \phi^{1,0}_{123}(t) \right\rangle \\[5pt]
(2,1) & \bC \left\langle \phi^{1,0}_{12}(t)\wedge\phi^{0,1}_{1}(t),\; \phi^{1,0}_{12}(t)\wedge\phi^{0,1}_{2}(t),\; \phi^{1,0}_{12}(t)\wedge\phi^{0,1}_{3}(t),\; \phi^{1,0}_{13}(t)\wedge\phi^{0,1}_{1}(t),\; \phi^{1,0}_{13}(t)\wedge \phi^{0,1}_{2}(t),\; \phi^{1,0}_{13}(t)\wedge\phi^{0,1}_{3}(t),\; \right.\\[5pt]
& \left. \phi^{1,0}_{23}(t)\wedge\phi^{0,1}_{1}(t),\; \phi^{1,0}_{23}(t)\wedge \phi^{0,1}_{2}(t),\; \phi^{1,0}_{23}(t)\wedge\phi^{0,1}_{3}(t),\; \phi^{1,0}_{1}(t)\wedge\varphi^{1,0}_{2}(t)\wedge\phi^{0,1}_{1}(t),\; \phi^{1,0}_{1}(t)\wedge\varphi^{1,0}_{2}(t)\wedge\varphi^{0,1}_{2}(t),\; \right. \\[5pt]
& \left. \phi^{1,0}_{1}(t)\wedge\varphi^{1,0}_{3}(t)\wedge\phi^{0,1}_{1}(t),\; \phi^{1,0}_{1}(t)\wedge\varphi^{1,0}_{3}(t)\wedge\varphi^{0,1}_{3}(t),\; \phi^{1,0}_{23}(t)\wedge\varphi^{0,1}_{2}(t),\; \phi^{1,0}_{23}(t)\wedge\varphi^{0,1}_{3}(t)
\right\rangle \\[5pt]
(1,2) & \bC \left\langle
\phi^{1,0}_{1}(t)\wedge\phi^{0,1}_{1}(t)\wedge\varphi^{0,1}_{2}(t),\; \varphi^{1,0}_{2}(t)\wedge\phi^{0,1}_{1}(t)\wedge\varphi^{0,1}_{2}(t),\; \phi_{3}^{1,0}(t)\wedge\phi^{0,1}_{12}(t),\; \phi^{1,0}_{1}(t)\wedge\phi^{0,1}_{1}(t)\wedge\varphi^{0,1}_{3}(t),\;\right. \\[5pt]
& \left.  \phi_{2}^{1,0}(t)\wedge\phi^{0,1}_{13}(t),\;  \varphi^{1,0}_{3}(t)\wedge\phi^{0,1}_{1}(t)\wedge\varphi^{0,1}_{3}(t),\;\phi^{1,0}_{1}(t)\wedge\phi^{0,1}_{23}(t),\;
\varphi^{1,0}_{2}(t)\wedge\varphi^{0,1}_{23}(t),\;\varphi^{1,0}_{3}(t)\wedge\varphi^{0,1}_{23}(t), \right. \\[5pt]
& \left. \phi^{1,0}_{1}(t)\wedge\phi^{0,1}_{12}(t),\; \phi^{1,0}_{2}(t)\wedge\phi^{0,1}_{12}(t),\; \phi^{1,0}_{1}(t)\wedge\phi^{0,1}_{13}(t),\; \phi^{1,0}_{3}(t)\wedge\phi^{0,1}_{13}(t),\; \phi^{1,0}_{2}(t)\wedge\phi^{0,1}_{23}(t),\; \phi^{1,0}_{3}(t)\wedge\phi^{0,1}_{23}(t) \right\rangle \\[5pt]
(0,3) & \bC \left\langle \phi^{0,1}_{123} (t)\right\rangle \\
\midrule[0.02em]
(3,1) & \bC \left\langle \phi^{1,0}_{123}(t)\wedge\phi^{0,1}_{1}(t),\; \phi^{1,0}_{123}(t)\wedge\phi^{0,1}_{2}(t),\; \phi^{1,0}_{123}(t)\wedge\phi^{0,1}_{3}(t),\;  \phi^{1,0}_{123}(t)\wedge\varphi^{0,1}_{2}(t),\; \phi^{1,0}_{123}(t)\wedge\varphi^{0,1}_{3} (t) \right\rangle \\[5pt]
(2,2) & \bC \left\langle \phi^{1,0}_{12}(t)\wedge\phi^{0,1}_{12}(t),\; \phi^{1,0}_{12}(t)\wedge\phi^{0,1}_{13}(t),\; \phi^{1,0}_{12}(t)\wedge\phi^{0,1}_{23}(t),\; \phi^{1,0}_{13}(t)\wedge\phi^{0,1}_{12}(t),\; \phi^{1,0}_{13}(t)\wedge\phi^{0,1}_{13}(t),\;  \right.\\[5pt]
& \left.
\phi^{1,0}_{13}(t)\wedge\phi^{0,1}_{23}(t) ,\; \phi^{1,0}_{23}(t)\wedge\phi^{0,1}_{12}(t),\; \phi^{1,0}_{23}(t)\wedge\phi^{0,1}_{13}(t),\;  \phi^{1,0}_{23}(t)\wedge\phi^{0,1}_{23}(t), \; \phi^{1,0}_{1}(t)\wedge \varphi^{1,0}_{2}(t)\wedge \phi_{1}^{0,1}(t)\wedge \varphi^{0,1}_{2}(t),\;\right.\\[5pt]
& \left. \phi_{1}^{1,0}(t)\wedge\varphi_{2}^{1,0}(t)\wedge\phi^{0,1}_{23}(t),\; \phi^{1,0}_{1}(t)\wedge \varphi^{1,0}_{3}(t)\wedge \phi_{1}^{0,1}(t)\wedge \varphi^{0,1}_{3}(t),\; \phi_{1}^{1,0}(t)\wedge \varphi_{3}^{1,0}(t)\wedge\phi^{0,1}_{23}(t), \right.\\[5pt]
& \left. \phi^{1,0}_{23}(t)\wedge\phi_{1}^{0,1}(t)\wedge\varphi_{2}^{0,1}(t),\; \phi^{1,0}_{23}(t)\wedge\phi_{1}^{0,1}(t)\wedge\varphi_{3}^{0,1}(t) \right\rangle \\[5pt]
(1,3) & \bC \left\langle \phi^{1,0}_{1}(t)\wedge\phi^{0,1}_{123}(t),\; \phi^{1,0}_{2}(t)\wedge\phi^{0,1}_{123}(t),\;\phi^{1,0}_{3}(t)\wedge\phi^{0,1}_{123}(t),\; \varphi^{1,0}_{2}(t)\wedge\phi^{0,1}_{123}(t),\;\varphi^{1,0}_{3}(t)\wedge\phi^{0,1}_{123}(t) \right\rangle \\
\midrule[0.02em]
(3,2) & \bC \left\langle \phi^{1,0}_{123}(t)\wedge\phi^{0,1}_{12}(t),\; \phi^{1,0}_{123}(t)\wedge\phi^{0,1}_{13}(t),\; \phi^{1,0}_{123}(t)\wedge\phi^{0,1}_{23}(t),\; 
\phi^{1,0}_{123}(t)\wedge\phi^{0,1}_{1}(t)\wedge\varphi^{0,1}_{2}(t),\; \phi^{1,0}_{123}(t)\wedge\phi^{0,1}_{1}(t)\wedge\varphi^{0,1}_{3}(t)  \right\rangle \\[5pt]
(2,3) & \bC \left\langle \phi^{1,0}_{12}(t)\wedge \phi^{0,1}_{123}(t),\; \phi^{1,0}_{13}(t)\wedge \phi^{0,1}_{123}(t),\; \phi^{1,0}_{23}(t)\wedge \phi^{0,1}_{123}(t),\; 
\phi^{1,0}_{1}(t)\wedge\varphi^{1,0}_{2}(t)\wedge \phi^{0,1}_{123}(t),\; \phi^{1,0}_{1}(t)\wedge \varphi^{1,0}_{3}(t)\wedge \phi^{0,1}_{123}(t)
 \right\rangle \\
\midrule[0.02em]
(3,3) & \bC \left\langle\phi^{1,0}_{123}(t)\wedge \phi^{0,1}_{123}(t) \right\rangle \\
\bottomrule
\end{tabular}
}}
\caption{The double-complex $C^{\bullet,\bullet}_\Gamma(t)$ for computing the Bott-Chern cohomology of the small deformations in case {\itshape (1)} of the holomorphically parallelizable Nakamura manifold $\left. \Gamma \middle\backslash G \right.$.}
\end{table}

\begin{table}[!hpt]\label{table:nak-def-1-delbar}
 \centering
\begin{tabular}{>{$\mathbf\bgroup}l<{\mathbf\egroup$} || >{$}l<{$}}
\toprule
\text{\bfseries case {\itshape (1)}} & H^{\bullet,\bullet} _{\bar\partial_{t}}(X) \\
\toprule
(0,0) & \bC \left\langle 1 \right\rangle \\
\midrule[0.02em]
(1,0) & \bC \left\langle \phi^{1,0}_{1}(t) \right\rangle  \\[5pt]
(0,1) & \bC \left\langle \phi^{0,1}_{1}(t) \right\rangle \\
\midrule[0.02em]
(2,0) & \bC \left\langle  \phi^{1,0}_{23}(t) \right\rangle \\[5pt]
(1,1) & \bC \left\langle \phi^{1,0}_{1}(t)\wedge\phi^{0,1}_{1}(t),\;  \phi^{1,0}_{2}(t)\wedge\phi^{0,1}_{3}(t),\;  \phi^{1,0}_{3}(t)\wedge\phi^{0,1}_{2}(t) \right\rangle \\[5pt]
(0,2) & \bC \left\langle  \phi^{0,1}_{23}(t) \right\rangle \\
\midrule[0.02em]
(3,0) & \bC \left\langle \phi^{1,0}_{123}(t) \right\rangle \\[5pt]
(2,1) & \bC \left\langle   \phi^{1,0}_{12}(t)\wedge\phi^{0,1}_{3}(t),\; \phi^{1,0}_{13}(t)\wedge \phi^{0,1}_{2}(t),\; \phi^{1,0}_{23}(t)\wedge\phi^{0,1}_{1}(t)\right\rangle \\[5pt]
(1,2) & \bC \left\langle \phi^{1,0}_{3}(t)\wedge\phi^{0,1}_{12}(t),\;  \phi^{1,0}_{2}(t)\wedge\phi^{0,1}_{13}(t),\;  \phi^{1,0}_{1}(t)\wedge\phi^{0,1}_{23}(t) \right\rangle \\[5pt]
(0,3) & \bC \left\langle \phi^{0,1}_{123} (t)\right\rangle \\
\midrule[0.02em]
(3,1) & \bC \left\langle \phi^{1,0}_{123}(t)\wedge\phi^{0,1}_{1} \right\rangle \\[5pt]
(2,2) & \bC \left\langle \phi^{1,0}_{12}(t)\wedge\phi^{0,1}_{13}(t),\;  \phi^{1,0}_{13}(t)\wedge\phi^{0,1}_{12}(t),\; 
 \phi^{1,0}_{23}(t)\wedge\phi^{0,1}_{23}(t) \right\rangle \\[5pt]
(1,3) & \bC \left\langle \phi^{1,0}_{1}(t)\wedge\phi^{0,1}_{123}(t) \right\rangle \\
\midrule[0.02em]
(3,2) & \bC \left\langle  \phi^{1,0}_{123}(t)\wedge\phi^{0,1}_{23}(t) \right\rangle \\[5pt]
(2,3) & \bC \left\langle\phi^{1,0}_{23}(t)\wedge \phi^{0,1}_{123}(t) 
 \right\rangle \\
\midrule[0.02em]
(3,3) & \bC \left\langle\phi^{1,0}_{123}(t)\wedge \phi^{0,1}_{123}(t) \right\rangle \\
\bottomrule
\end{tabular}
\caption{The harmonic representatives of the Dolbeault cohomology of the small deformations in case {\itshape (1)}, which are given by $t\,\frac{\partial}{\partial z_{1}}\otimes \de \bar z_{1}$, of the holomorphically parallelizable Nakamura manifold, with respect to the Hermitian metric $g_{t}:=\phi^{1,0}_{1}(t)\odot\phi^{0,1}_{1}(t)+\phi^{1,0}_{2}(t)\odot\varphi^{0,1}_{2}(t)+\phi^{1,0}_{3}(t)\odot\varphi^{0,1}_{3}(t)$.}
\end{table}

\begin{table}[!hpt]\label{table:nak-def-1-bc}
 \centering
\begin{tabular}{>{$\mathbf\bgroup}l<{\mathbf\egroup$} || >{$}l<{$}}
\toprule
\text{\bfseries case {\itshape (1)}} & H^{\bullet,\bullet}_{{BC}_{J_t}}(X) \\
\toprule
(0,0) & \bC \left\langle 1 \right\rangle \\
\midrule[0.02em]
(1,0) & \bC \left\langle \phi^{1,0}_{1}(t) \right\rangle  \\[5pt]
(0,1) & \bC \left\langle \phi^{0,1}_{1}(t) \right\rangle \\
\midrule[0.02em]
(2,0) & \bC \left\langle  \phi^{1,0}_{23}(t) \right\rangle \\[5pt]
(1,1) & \bC \left\langle \phi^{1,0}_{1}(t)\wedge\phi^{0,1}_{1}(t),\;  \phi^{1,0}_{2}(t)\wedge\phi^{0,1}_{3}(t),\;  \phi^{1,0}_{3}(t)\wedge\phi^{0,1}_{2}(t) \right\rangle \\[5pt]
(0,2) & \bC \left\langle  \phi^{0,1}_{ 23}(t) \right\rangle \\
\midrule[0.02em]
(3,0) & \bC \left\langle \phi^{1,0}_{123}(t) \right\rangle \\[5pt]
(2,1) & \bC \left\langle   \phi^{1,0}_{12}(t)\wedge\phi^{0,1}_{3}(t),\; \phi^{1,0}_{13}(t)\wedge \phi^{0,1}_{2}(t),\; \phi^{1,0}_{23}(t)\wedge\phi^{0,1}_{1}(t)\right\rangle \\[5pt]
(1,2) & \bC \left\langle \phi^{1,0}_{3}(t)\wedge\phi^{0,1}_{12}(t),\;  \phi^{1,0}_{2}(t)\wedge\phi^{0,1}_{13}(t),\;  \phi^{1,0}_{1}(t)\wedge\phi^{0,1}_{23}(t) \right\rangle \\[5pt]
(0,3) & \bC \left\langle \phi^{0,1}_{123} (t)\right\rangle \\
\midrule[0.02em]
(3,1) & \bC \left\langle \phi^{1,0}_{123}(t)\wedge\phi^{0,1}_{1}(t) \right\rangle \\[5pt]
(2,2) & \bC \left\langle \phi^{1,0}_{12}(t)\wedge\phi^{0,1}_{13}(t),\;  \phi^{1,0}_{13}(t)\wedge\phi^{0,1}_{12}(t),\; 
 \phi^{1,0}_{23}(t)\wedge\phi^{0,1}_{23}(t) \right\rangle \\[5pt]
(1,3) & \bC \left\langle \phi^{1,0}_{1}(t)\wedge\phi^{0,1}_{123}(t) \right\rangle \\
\midrule[0.02em]
(3,2) & \bC \left\langle  \phi^{1,0}_{123}(t)\wedge\phi^{0,1}_{23}(t) \right\rangle \\[5pt]
(2,3) & \bC \left\langle\phi^{1,0}_{23}(t)\wedge \phi^{0,1}_{123}(t) \right\rangle \\
\midrule[0.02em]
(3,3) & \bC \left\langle\phi^{1,0}_{123}(t)\wedge \phi^{0,1}_{123}(t) \right\rangle \\
\bottomrule
\end{tabular}
\caption{The harmonic representatives of the Bott-Chern cohomology of the small deformations in case {\itshape (1)}, which are given by $t\,\frac{\partial}{\partial z_{1}}\otimes \de \bar z_{1}$, of the holomorphically parallelizable Nakamura manifold, with respect to the Hermitian metric $g_{t}:=\phi^{1,0}_{1}(t)\odot\phi^{0,1}_{1}(t)+\phi^{1,0}_{2}(t)\odot\varphi^{0,1}_{2}(t)+\phi^{1,0}_{3}(t)\odot\varphi^{0,1}_{3}(t)$.}
\end{table}

\begin{table}[!hpt]\label{table:nak-def-2-delbar}
 \centering
\begin{tabular}{>{$\mathbf\bgroup}l<{\mathbf\egroup$} || >{$}l<{$}}
\toprule
\text{\bfseries case {\itshape (2)}} & H^{\bullet,\bullet} _{\bar\partial_{t}}(X)\\
\toprule
(0,0) & \bC \left\langle 1 \right\rangle \\
\midrule[0.02em]
(1,0) & 0 \\[5pt]
(0,1) & \bC \left\langle \phi^{0,1}_{1}(t),\; \phi^{0,1}_{3}(t)\right\rangle \\
\midrule[0.02em]
(2,0) & \bC \left\langle \phi^{1,0}_{12}(t),\; \phi^{1,0}_{23}(t) \right\rangle \\[5pt]
(1,1) & \bC \left\langle  \phi^{1,0}_{1}(t)\wedge\phi^{0,1}_{2}(t),\; \phi^{1,0}_{3}(t)\wedge\phi^{0,1}_{2}(t) \right\rangle \\[5pt]
(0,2) & \bC \left\langle  \phi^{0,1}_{13}(t) \right\rangle \\
\midrule[0.02em]
(3,0) & 0 \\[5pt]
(2,1) & \bC \left\langle \phi^{1,0}_{12}(t)\wedge\phi^{0,1}_{1}(t),\; \phi^{1,0}_{12}(t)\wedge\phi^{0,1}_{3}(t),\; \phi^{1,0}_{23}(t)\wedge\phi^{0,1}_{1}(t),\;  \phi^{1,0}_{23}(t)\wedge\phi^{0,1}_{3}(t) \right\rangle \\[5pt]
(1,2) & \bC \left\langle \phi^{1,0}_{3}(t)\wedge\phi^{0,1}_{12}(t),\;  \phi^{1,0}_{1}(t)\wedge\phi^{0,1}_{23}(t),\; \phi^{1,0}_{1}(t)\wedge\phi^{0,1}_{12}(t),\; \phi^{1,0}_{3}(t)\wedge\phi^{0,1}_{23}(t) \right\rangle \\[5pt]
(0,3) & 0 \\
\midrule[0.02em]
(3,1) & \bC \left\langle \phi^{1,0}_{123}(t)\wedge\phi^{0,1}_{2} \right\rangle \\[5pt]
(2,2) & \bC \left\langle \phi^{1,0}_{12}(t)\wedge\phi^{0,1}_{13}(t),\;  \phi^{1,0}_{23}(t)\wedge\phi^{0,1}_{13}(t)   \right\rangle \\[5pt]
(1,3) & \bC \left\langle \phi^{1,0}_{1}(t)\wedge\phi^{0,1}_{123}(t),\; \phi^{1,0}_{3}(t)\wedge\phi^{0,1}_{123}(t) \right\rangle \\
\midrule[0.02em]
(3,2) & \bC \left\langle  \phi^{1,0}_{123}(t)\wedge\phi^{0,1}_{12}(t),\; \phi^{1,0}_{123}(t)\wedge\phi^{0,1}_{23}(t)  \right\rangle \\[5pt]
(2,3) & 0 \\
\midrule[0.02em]
(3,3) & \bC \left\langle\phi^{1,0}_{123}(t)\wedge \phi^{0,1}_{123}(t) \right\rangle \\
\bottomrule
\end{tabular}
\caption{The harmonic representatives of the Dolbeault cohomology of the small deformations in case {\itshape (2)}, which are given by $t\,\frac{\partial}{\partial z_{1}}\otimes \esp^{z_1}\de \bar z_{3}$, of the holomorphically parallelizable Nakamura manifold, with respect to the Hermitian metric $g_{t}:=\phi^{1,0}_{1}(t)\odot\phi^{0,1}_{1}(t)+\phi^{1,0}_{2}(t)\odot\varphi^{0,1}_{2}(t)+\phi^{1,0}_{3}(t)\odot\varphi^{0,1}_{3}(t)$.}
\end{table}

\begin{table}[!hpt]\label{table:nak-cohomologies}
 \centering
\begin{tabular}{>{$\mathbf\bgroup}c<{\mathbf\egroup$} || >{$}c<{$} | >{$}c<{$} >{$}c<{$} | >{$}c<{$} | >{$}c<{$} >{$}c<{$} | >{$}c<{$} | >{$}c<{$}}
\toprule
\textnormal{$\dim_\bC H_{\sharp}^{\bullet,\bullet}$} & \multicolumn{3}{c|}{\text{\bfseries Nakamura}} & \multicolumn{3}{c|}{\text{\bfseries case {\itshape (1)}}} & \multicolumn{2}{c}{\text{\bfseries case {\itshape (2)}}} \\
& dR & \bar\partial & BC & dR & \bar\partial & BC & dR & \bar\partial \\
\toprule
(0,0) & 1 & 1 & 1 & 1 &  1 & 1 & 1& 1 \\
\midrule[0.02em]
(1,0) & \multirow{2}{*}{2} & 3 & 1 & \multirow{2}{*}{2} & 1 & 1 &\multirow{2}{*}{2}& 0 \\[5pt]
(0,1) & & 3 & 1 & & 1 &  1 && 2 \\
\midrule[0.02em]
(2,0) & \multirow{3}{*}{5} & 3 & 3 & \multirow{3}{*}{5} & 1 & 1 & \multirow{3}{*}{5} & 2 \\[5pt]
(1,1) & & 9 & 7 & & 3 &  3 & &2 \\[5pt]
(0,2) & & 3 & 3 & & 1 &  1 & &1 \\
\midrule[0.02em]
(3,0) & \multirow{4}{*}{8} & 1 & 1 & \multirow{4}{*}{8} & 1 & 1 & \multirow{4}{*}{8} & 0 \\[5pt]
(2,1) & & 9 & 9 & & 3 &  3 & & 4 \\[5pt]
(1,2) & & 9 & 9 & & 3 &  3 & & 4 \\[5pt]
(0,3) & & 1 & 1 & & 1 &  1 & & 0 \\
\midrule[0.02em]
(3,1) & \multirow{3}{*}{5} & 3 & 3 & \multirow{3}{*}{5} & 1 & 1 & \multirow{3}{*}{5} & 1 \\[5pt]
(2,2) & & 9 & 11 & & 3 & 3 & & 2 \\[5pt]
(1,3) & & 3 & 3 & & 1 &  1 & &2 \\
\midrule[0.02em]
(3,2) & \multirow{2}{*}{2} & 3 & 5 & \multirow{2}{*}{2} & 1 & 1 & \multirow{2}{*}{2}& 2 \\[5pt]
(2,3) & & 3 & 5 & & 1 & 1 & & 0 \\
\midrule[0.02em]
(3,3) & 1 & 1 & 1 & 1 & 1 & 1&1 & 1 \\
\bottomrule
\end{tabular}
\caption{Summary of the dimensions of the cohomologies of the holomorphically parallelizable Nakamura manifold $X$, \cite[Example 2.25]{angella-kasuya-1}, and of its small deformations in case {\itshape (1)} and {\itshape (2)}, given, respectively, by $t\,\frac{\partial}{\partial z_{1}}\otimes \de \bar z_{1}$ and by $t\,\frac{\partial}{\partial z_{1}}\otimes \esp^{z_1}\de \bar z_{3}$.}
\end{table}

\medskip

Straightforwardly, (e.g., from Table \ref{table:nak-cohomologies} and by \cite[Theorem B]{angella-tomassini-3},) we get the following result. (See \cite{kasuya-hodge} for other examples of non-K\"ahler solvmanifolds satisfying the $\partial\overline{\partial}$-Lemma.)

\begin{proposition}
Consider the holomorphically parallelizable Nakamura manifold $\left( X,\, J_0\right)$, and its small deformations $\left\{J_t\right\}_{t\in B}$ as in case {\itshape (1)} or case {\itshape (2)}. Then
\begin{enumerate}[{\itshape (i)}]
\item the deformations $\left(X,\, J_{t}\right)$ as in case {\itshape (1)} satisfy the $\partial\bar\partial$-Lemma.
\item the deformations $\left(X,\, J_{t}\right)$ as in case {\itshape (2)} satisfy the $E_{1}$-degeneration of the Hodge and Fr\"olicher spectral sequences, but do not satisfy the $\partial\bar\partial$-Lemma.
\end{enumerate}
\end{proposition}

\section{Example: Sawai and Yamada generalized manifolds}\label{SYE}
In this section, we study the cohomology of the generalized examples introduced and studied by H. Sawai and T. Yamada in \cite{sawai-yamada} in order to generalize Ch. Benson and C.~S. Gordon manifold \cite{benson-gordon}.

\medskip

Following \cite{sawai-yamada}, let $\mathfrak{n}$ be a complex nilpotent Lie algebra.
We assume that
$$ \mathfrak{n} \;=\; \bC\left\langle Y_{1},\dots ,Y_{\ell},Y_{\ell+1}, \dots ,Y_{m} \right\rangle $$
so that $[\mathfrak{n},\mathfrak{n}]=\bC\langle Y_{\ell+1}, \dots,Y_{m}\rangle$ and $[Y_{i},Y_{j}]=C_{ij}^{k}\,Y_{k}$ for some $C_{ij}^{k}\in \Z$, varying $i,j,k\in\{1,\ldots,m\}$.
Define
\[\tilde{\mathfrak{n}} \;:=\; \bC\langle Y_{1,1},\dots ,Y_{1,\ell}, Y_{1,\ell+1} ,Y_{1,m}\rangle\oplus \bC\langle Y_{2,1},\dots ,Y_{2,\ell},Y_{2,\ell+1} ,Y_{2,m}\rangle\]
where $\bC\langle Y_{1,1},\dots ,Y_{1,\ell}, Y_{1,\ell+1},\dots ,Y_{1,m}\rangle\cong \bC\langle Y_{2,1},\dots ,Y_{2,\ell}, Y_{2,\ell+1},\dots ,Y_{2,m}\rangle\cong \mathfrak{n}$.
Consider the semi-direct product $\g:=\bC\langle X\rangle\ltimes \tilde{\mathfrak{n}}$ given by 
\[ \left[X, Y_{1,j}\right] \;:=\; k_{j}\, Y_{1,j}\;, \qquad \left[X, Y_{2,j}\right] \;:=\; -k_{j}\, Y_{2,j} \]
where $\left\{k_{j}\right\}_j\subset \N\setminus\{0\}$ is such that the Jacobi identity holds.

Let $G=\bC\ltimes \tilde N$ be the connected simply-connected complex Lie group corresponding to $\g$.
Then we have
\[ G \;=\; \left\{\left(z, \left(
\begin{array}{cc}
w_{1,1}\\
w_{2,1}  
\end{array}
\right) ,\dots , \left(
\begin{array}{cc}
w_{1,m}\\
w_{2,m}  
\end{array}
\right)\right)
\;:\;
z, w_{1,j}, w_{2,j} \in \bC\right\}
\]
with the product
\begin{eqnarray*}
\lefteqn{
\left(z, \left(
\begin{array}{cc}
w_{1,1}\\
w_{2,1}  
\end{array}
\right) ,\dots , \left(
\begin{array}{cc}
w_{1,m}\\
w_{2,m}  
\end{array}
\right)\right)\cdot \left(z^{\prime}, \left(
\begin{array}{cc}
w^{\prime}_{1,1}\\
w^{\prime}_{2,1}  
\end{array}
\right) ,\dots , \left(
\begin{array}{cc}
w^{\prime}_{1,m}\\
w^{\prime}_{2,m}  
\end{array}
\right)\right)
}\\
&=&
\left( z+z^{\prime}, \left(
\begin{array}{cc}
f_{1,1}(z,w_{1,1},\ldots,w_{1,m},w_{1,1}^{\prime},\ldots,w_{1,m}^{\prime})\\
f_{2,1}(z,w_{2,1},\ldots,w_{2,m},w_{2,1}^{\prime},\ldots,w_{2,m}^{\prime})
\end{array}
\right), \dots ,\right.\\[5pt]
&& \left. \left(
\begin{array}{cc}
f_{1,m}(z,w_{1,1},\ldots,w_{1,m},w_{1,1}^{\prime},\ldots,w_{1,m}^{\prime})\\
f_{2,m}(z,w_{2,1},\ldots,w_{2,m},w_{2,1}^{\prime},\ldots,w_{2,m}^{\prime})
\end{array}
\right) \right) \;,
\end{eqnarray*}
for certain functions $f_{1,1},\ldots,f_{1,m},f_{2,1},\ldots,f_{2,m}$, see \cite[Section 2]{sawai-yamada}.

Take a unimodular matrix $B\in \mathrm{SL}(2,\Z)$ with distinct positive eigenvalues $\lambda$ and $\lambda^{-1}$, and set $a:=\log \lambda$.
Consider
\begin{eqnarray*}
\Gamma &:=& \left\{\left(a\,s+2\pi\sqrt{-1}\,t, \left(
\begin{array}{cc}
w_{1,1}+\lambda\, w_{2,1} \\
w_{1,1}  +\lambda^{-1}\, w_{2,1} 
\end{array}
\right) ,\dots , \left(
\begin{array}{cc}
w_{1,m}+\lambda\, w_{2,m} \\
w_{1,m}  +\lambda^{-1}\, w_{2,m} 
\end{array}
\right)\right) \right. \\[5pt]
&& \left. \;:\; s,t\in \Z, \; w_{1,j},w_{2,j}\in \Z+\sqrt{-1}\,\Z\right\}\;.
\end{eqnarray*}
Then, as H. Sawai and T. Yamada proved in \cite[Theorem 2.1]{sawai-yamada}, $\Gamma$ is a lattice in $G$.
Hence we have $\Gamma=(a\,\Z+2\pi\sqrt{-1}\,\Z)\ltimes \Gamma^{\prime\prime}$ such that $\Gamma^{\prime\prime}$ is a lattice in $\tilde N$.

\medskip

Let $\left\{ y_{1,1},\dots ,y_{1,\ell}, y_{1,\ell+1},\dots ,y_{1,m}, y_{2,1},\dots ,y_{2,\ell}, y_{2,\ell+1},\dots ,y_{2,m} \right\}$ be the dual basis of the space $\left(\tilde{\mathfrak{n}}^{1,0}\right)^*$ of the left-invariant $(1,0)$-forms on $\tilde N$.
Then, by the assumption, we have $\de y_{1,j}=\de y_{2,j}=0$ for $1\le j\le \ell$.
The space $\left(\g^{1,0}\right)^*$ of the left-invariant $(1,0)$-forms on $G$ is given by
\[\left(\g^{1,0}\right)^* \;=\; \bC\left\langle \de z, \esp^{-k_{1}z}y_{1,1}, \dots ,\esp^{-k_{m}z}y_{1,m}, \esp^{k_{1}z}y_{2,1},\dots ,\esp^{k_{m}z}y_{2,m}\right\rangle \;.
\]
Consider
\begin{eqnarray*}
B_{\Gamma}^{\bullet,\bullet}
&:=& \wedge^{\bullet,\bullet} \bC\left\langle \de z, \esp^{-k_{1}z}y_{1,1}, \dots ,\esp^{-k_{m}z}y_{1,m}, \esp^{k_{1}z}y_{2,1},\dots ,\esp^{k_{m}z}y_{2,m}\right\rangle\\[5pt]
&&\otimes \,\bC\left\langle \de\bar z, \esp^{-k_{1}z}\bar y_{1,1}, \dots ,\esp^{-k_{m}z}\bar y_{1,m}, \esp^{k_{1}z}\bar y_{2,1},\dots ,\esp^{k_{m}z}\bar y_{2,m}\right\rangle \;.
\end{eqnarray*}
Then we have
\[H^{\bullet,\bullet}_{\bar\partial}(B_{\Gamma}^{\bullet,\bullet})\cong H^{\bullet,\bullet}_{\bar\partial}(\left. \Gamma \middle\backslash G \right.).
\]
We consider deformations $\{J_{t}\}_{t\in B}$ over a ball $B\subset \bC$ given by: 
$$ t\,\frac{\partial}{\partial z}\otimes \esp^{k_{1}z}\bar y_{2,1} \;\in\; H^{0,1}(\left. \Gamma \middle\backslash G \right.; T^{1,0}\left. \Gamma \middle\backslash G \right.) \;. $$

To compute the Dolbeault cohomology of $\left(\left. \Gamma \middle\backslash G \right., J_t\right)$, consider the forms defined in Table \ref{table:sawai-yamada-b-gen}.

\begin{table}[!hpt]\label{table:sawai-yamada-b-gen}
 \centering
\begin{tabular}{>{$}l<{$} || >{$}l<{$}}
\toprule
\mathbf{\psi} & \mathbf{\de\psi} \\
\toprule
\phi^{1,0}_{0}(t):=\de z-t\,\esp^{k_{1}z}\bar y_{2,1} & \de\phi^{1,0}_{0}(t) = -t\,k_{1} \phi^{1,0}_{0}(t)\wedge \phi^{0,1}_{2,1}(t) \\[5pt]
\phi^{1,0}_{1,j}(t):=\esp^{-k_{j}z}y_{1,j} & \de\phi^{1,0}_{1,j}(t) = -k_{j}\phi_{0}^{1,0}(t)\wedge \phi_{1,j}^{1,0}(t)+t\,k_{j} \phi^{1,0}_{1,j}(t)\wedge \phi^{0,1}_{2,1}(t)+\esp^{-k_{j} z}\de y_{1,j} \\[5pt]
\phi^{1,0}_{2,j}(t):=\esp^{k_{j}z}y_{2,j} & \de\phi^{1,0}_{2,j}(t) = k_{j}\phi_{0}^{1,0}(t)\wedge \phi_{2,j}^{1,0}(t)-t\,k_{j} \phi^{1,0}_{2,j}(t)\wedge \phi^{0,1}_{2,1}(t)+\esp^{k_{j} z}\de y_{2,j} \\
\midrule[0.02em]
\varphi^{1,0}_{1,j}(t):=e^{-k_{j}\bar z}y_{1,j} & \de\varphi^{1,0}_{1,j}(t) = -k_{j}\phi_{0}^{0,1}(t)\wedge \varphi_{1,j}^{1,0}(t)+\bar t\,k_{j} \varphi^{1,0}_{1,j}(t)\wedge \varphi^{1,0}_{2,1}(t)+ \esp^{-k_{j}\bar z}\de y_{1,j} \\[5pt]
\varphi^{1,0}_{2,j}(t):=e^{k_{j}\bar z}y_{2,j} & \de\varphi^{1,0}_{2,j}(t) = k_{j}\phi_{0}^{0,1}(t)\wedge \varphi_{2,j}^{1,0}(t)-\bar t\,k_{j} \varphi^{1,0}_{2,j}(t)\wedge \varphi^{1,0}_{2,1}(t)+\esp^{k_{j}\bar z}\de y_{2,j} \\
\midrule[0.02em] \midrule[0.02em]
\phi^{0,1}_{0}(t):=\de\bar z-\bar t\,\esp^{k_{1}\bar z} y_{2,1} & \de\phi^{0,1}_{0}(t) = -t\,k_{1} \phi^{0,1}_{0}(t)\wedge \varphi^{1,0}_{2,1}(t) \\[5pt]
\phi^{0,1}_{1,j}(t):=\esp^{-k_{j}z}\bar y_{1,j} & \de\phi^{0,1}_{1,j}(t) = -k_{j}\phi_{0}^{1,0}(t)\wedge \phi_{1,j}^{0,1}(t)+t\,k_{j} \phi^{0,1}_{1,j}(t)\wedge \phi^{0,1}_{2,1}(t)+ \esp^{-k_{j} z}\de\bar y_{1,j} \\[5pt]
\phi^{0,1}_{2,j}(t):=\esp^{k_{j}z}\bar y_{2,j} & \de\phi^{0,1}_{2,j}(t) = k_{j}\phi_{0}^{1,0}(t)\wedge \phi_{2,j}^{0,1}(t)-t\,k_{j} \phi^{0,1}_{2,j}(t)\wedge \phi^{0,1}_{2,1}(t)+\esp^{k_{j}z}\de\bar y_{2,j} \\
\midrule[0.02em]
\varphi^{0,1}_{1,j}(t):=\esp^{-k_{j}\bar z}\bar y_{1,j} & \de\varphi^{0,1}_{1,j}(t) = -k_{j}\phi_{0}^{0,1}(t)\wedge \varphi_{1,j}^{0,1}(t)+\bar t\,k_{j} \varphi^{0,1}_{1,j}(t)\wedge \varphi^{1,0}_{2,1}(t)+\esp^{-k_{j}\bar z}\de\bar y_{1,j} \\[5pt]
\varphi^{0,1}_{2,j}(t):=\esp^{k_{j}\bar z}\bar y_{2,j} & \de\varphi^{0,1}_{2,j}(t) = k_{j}\phi_{0}^{0,1}(t)\wedge \varphi_{2,j}^{0,1}(t)-\bar t\,k_{j} \varphi^{0,1}_{2,j}(t)\wedge \phi^{1,0}_{2,1}(t)+\esp^{k_{j}\bar z}\de\bar y_{2,j} \\
\bottomrule
\end{tabular}
\caption{Definitions for setting the generators of the complex $B^{\bullet,\bullet}_\Gamma(t)$, see \eqref{eq:sawai-yamada-def-b}, for the deformations induced by $t\,\frac{\partial}{\partial z}\otimes \esp^{k_{1}z}\bar y_{2,1}$, of the Sawai and Yamada generalized manifold $\left. \Gamma \middle\backslash G \right.$.}
\end{table}

More precisely, by applying Theorem \ref{thm:dolb-deformations} to the double-complex
\begin{eqnarray}\label{eq:sawai-yamada-def-b}
B_{\Gamma}^{\bullet,\bullet}(t)
&=& \wedge^{\bullet,\bullet} \bC\left\langle \phi^{1,0}_{0}(t), \phi^{1,0}_{1,1}(t), \dots ,\phi^{1,0}_{1,m}(t), \phi^{1,0}_{2,1}(t),\dots ,\phi^{1,0}_{2,m}(t)\right\rangle\\[5pt]
&& \otimes \, \bC\left\langle \phi^{0,1}_{0}(t),\phi^{0,1}_{1,1}(t), \dots ,\phi^{0,1}_{1,m}(t), \phi^{0,1}_{2,1}(t),\dots ,\phi^{0,1}_{2,m}(t)\right\rangle \nonumber
\end{eqnarray}
and to the $J_t$-Hermitian metric
\[ g_{t} \;:=\; \phi^{1,0}_{0}(t)\odot\phi^{0,1}_{0}(t)+\sum_{j=1}^{m} \phi^{1,0}_{1,j}(t)\odot \varphi^{0,1}_{1,j}(t)+\sum_{j=1}^{m} \phi^{1,0}_{2,j}(t)\odot \varphi^{0,1}_{2,j}(t) \;,
\]
since $(B_{\Gamma}^{\bullet,\bullet}(t), \bar\partial_{t})$ is a sub-complex of $(\wedge^{\bullet,\bullet}(\left. \Gamma \middle\backslash G \right.), \bar\partial_{t})$ and 
$\bar*_{t}(B_{\Gamma}^{\bullet,\bullet}(t))\subseteq B_{\Gamma}^{2m+1-\bullet,2m+1-\bullet}(t)$, then we have
\[ H^{\bullet,\bullet}_{\bar\partial_{t}}(B_{\Gamma}^{\bullet,\bullet}(t)) \;\cong\; H^{\bullet,\bullet}_{\bar\partial_{t}}(\left. \Gamma \middle\backslash G \right.) \;.\]
By simple computations we have the following result.

\begin{proposition}\label{prop:sawai-yamada}
Consider the Sawai and Yamada generalized manifold $X = \left. \Gamma \middle\backslash G \right.$ of complex dimension $2m+1$, and its small deformations $\left\{J_t\right\}_{t\in B \subset \bC}$ induced by $t\,\frac{\partial}{\partial z}\otimes \esp^{k_{1}z}\bar y_{2,1} \in H^{0,1}(X; T^{1,0}X)$. Then
\[ \dim H^{1,0}_{\bar\partial_{t}}(X)\;=\; 0 \qquad \text{ and } \qquad \dim H^{2m+1,0}_{\bar\partial_{t}}(X)\;=\;0 \;.\]
\end{proposition}

\begin{remark}{\rm
 In \cite{kasuya-mathz, kasuya-holpar, console-fino-kasuya}, structures of holomorphic fibre bundles over complex tori with nilmanifold-fibres play a very important role for computing the Dolbeault cohomology of certain solvmanifolds. But, by Proposition \ref{prop:sawai-yamada}, such deformed complex solvmanifolds are not holomorphic fibre bundles over complex tori. Hence they provide new examples of ``Dolbeault-cohomologically-computable'' complex solvmanifolds.
}\end{remark}

\section{Closedness and openness under holomorphic deformation}

We recall that a property $\mathcal{P}$ concerning complex manifolds is called \emph{open} under holomorphic deformations if, whenever it holds for a compact complex manifold $X$, it holds also for any small deformations of $X$. It is called \emph{(Zariski-)closed} (simply, \emph{closed}) if, for any family $\left\{X_t\right\}_{t\in\Delta}$ of compact complex manifolds such that $\mathcal{P}$ holds for any $t\in\Delta\setminus\{0\}$ in the punctured-disk, then $\mathcal{P}$ holds also for $X_0$.

It is known that the $\partial\bar\partial$-Lemma is open under holomorphic deformations, see, e.g., \cite[Proposition 9.21]{voisin}, or \cite[Theorem 5.12]{wu}, or \cite[\S B]{tomasiello}, or \cite[Corollary 2.7]{angella-tomassini-3}.
In \cite[Theorem 2.20]{angella-kasuya-1}, the authors proved that the $\partial\bar\partial$-Lemma is not \emph{strongly-closed} under holomorphic deformations, namely, there exists a family $\left\{X_t\right\}_{t\in \Delta}$ of compact complex manifolds and a sequence $\left\{t_k\right\}_{k\in\N}\subset \Delta$ converging to $0\in \Delta$ in the Euclidean topology of $\Delta$ such that $X_{t_k}$ satisfies the $\partial\bar\partial$-Lemma and $X_0$ does not; more precisely, in \cite[Example 2.17]{angella-kasuya-1}, $X_0$ is the completely-solvable Nakamura manifold.

We prove now that the $\partial\bar\partial$-Lemma is also non-(Zariski-)closed. Indeed, consider the holomorphically parallelizable Nakamura manifold $\left. \Gamma \middle\backslash G \right.$ and its small deformations as in Section \ref{naka}.
While $\left. \Gamma \middle\backslash G \right.$ does not satisfy the $E_{1}$-degeneration of the Hodge and Fr\"olicher spectral sequences, deformations as in case {\itshape (1)} and {\itshape (2)} do. While $\left. \Gamma \middle\backslash G \right.$ does not satisfy the $\partial\bar\partial$-Lemma, deformations as in case {\itshape (1)} do. Hence we get the following result.

\begin{corollary}\label{cor:non-closedness}
The properties of $E_{1}$-degeneration of the Hodge and Fr\"olicher spectral sequences and the $\partial\bar\partial$-Lemma are not closed under holomorphic deformations.
\end{corollary}

The non-closedness of the property of $E_{1}$-degeneration of the Hodge and Fr\"olicher spectral sequences was firstly proven by M.~G. Eastwood and M.~A. Singer in \cite[Theorem 5.4]{eastwood-singer}, by considering twistor spaces.

\begin{remark}{\rm
Note that the small deformations $(\left. \Gamma \middle\backslash G \right.,\, J_{t})$ as in case {\itshape (1)} of the holomorphically parallelizable Nakamura manifold $\left. \Gamma \middle\backslash G \right.$ provide examples of compact complex manifolds that are not in {\em Fujiki class $\mathcal C$}, \cite{fujiki}, but satisfy the $\partial\bar\partial$-Lemma. This follows from \cite[Theorem 2.3]{chiose}. See also \cite[Theorem 9]{arapura}, or \cite[Theorem 3.3]{angella-kasuya-4}.

This is in accord with the conjectures that the property of being Mo\v\i\v shezon is closed under holomorphic deformations, see \cite{popovici-moishezon}, and that the Fujiki class $\mathcal C$ is closed under holomorphic deformations, \cite[Standard Conjecture 1.17]{popovici-annsns}, (compare also \cite[Question 1.5]{popovici-inventiones}).
}\end{remark}

\bibliographystyle{amsalpha}


\end{document}